\documentclass[letterpaper,11pt]{amsart}

\usepackage[margin=1.4in]{geometry}
\usepackage{amscd, amssymb}
\usepackage{amsmath,amscd}
\usepackage{mathtools}
\usepackage[driverfallback=hypertex]{hyperref}
\usepackage{graphics}
\usepackage{graphicx}
\usepackage{color}
\usepackage{bm}
\usepackage{enumitem}
\setlist[enumerate,1]{label={(\alph*)}}
\setlist[enumerate,2]{label={(\roman*)}}

\newif\ifdraft

\ifdraft
\usepackage{draftwatermark}
\SetWatermarkText{DRAFT}
\SetWatermarkScale{2}
\fi

\newtheorem{thm}{Theorem}[section]
\newtheorem{prop}[thm]{Proposition}

\newtheorem{lem}[thm]{Lemma}

\newtheorem{clm}[thm]{Claim}

\theoremstyle{definition}
\newtheorem{definition}[thm]{Definition}

\newtheorem*{notation*}{Notation}

\theoremstyle{remark}
\newtheorem{rmk}[thm]{Remark}

\newtheorem{obs}[thm]{Observation}

\newcommand{\ignore}[1]{}
\newcommand{\R}{\mathbb R}

\newcommand{\N}{\mathbb N}
\newcommand{\Prob}{{\mathbb{P}}}

\newcommand{\mA}{{\mathcal{A}}}
\newcommand{\mB}{{\mathcal{B}}}
\newcommand{\mF}{{\mathcal{F}}}

\newcommand{\mH}{{\mathcal{H}}}

\newcommand{\mT}{{\mathcal{T}}}
\newcommand{\E}{{\mathbb{E}}}

\newcommand{\oone}{{o \left(1\right)}}
\newcommand{\oneoone}{{\left( 1 \pm \oone \right)}}
\newcommand{\omegaone}{{\omega \left(1\right)}}
\newcommand{\termdefine}[1]{\textbf{#1}}

\newcommand{\tfr}{{T_{freeze}}}
\newcommand{\given}{{|}}

\newcommand{\alphadef}{{\beta/100}}
\newcommand{\betadef}{{\varepsilon/10}}

\DeclareMathOperator{\deg1}{deg}

\DeclareMathOperator{\girth}{girth}
\DeclareMathOperator{\diam}{diam}

\begin{document}
\title{A randomized construction of high girth regular graphs}
\author{Nati Linial}\thanks{Supported by Israel Science Foundation grant 659/18.}
\address{Department of Computer Science, The Hebrew University of Jerusalem, Jerusalem 91904, Israel}
\email{nati@cs.huji.ac.il}
\author{Michael Simkin}
\address{Institute of Mathematics and Federmann Center for the Study of Rationality, The Hebrew University of Jerusalem, Jerusalem 91904, Israel}
\email{menahem.simkin@mail.huji.ac.il}

\begin{abstract}
	We describe a new random greedy algorithm for generating regular graphs of high girth: Let $k\geq 3$ and $c \in (0,1)$ be fixed. Let $n \in \mathbb{N}$ be even and set $g = c \log_{k-1} (n)$. Begin with a Hamilton cycle $G$ on $n$ vertices. As long as the smallest degree $\delta (G)<k$, choose, uniformly at random, two vertices $u,v \in V(G)$ of degree $\delta(G)$ whose distance is at least $g-1$. If there are no such vertex pairs, abort. Otherwise, add the edge $uv$ to $E(G)$.
	
	We show that with high probability this algorithm yields a $k$-regular graph with girth at least $g$. Our analysis also implies that there are $\left( \Omega (n) \right)^{kn/2}$ labeled $k$-regular $n$-vertex graphs with girth at least $g$.
\end{abstract}

\maketitle

\pagestyle{plain}

\section{Introduction}

The \termdefine{girth} of a graph is the length of its shortest cycle. It is a classical challenge to determine $g(k,n)$, the largest possible girth of $k$-regular graphs with $n$ vertices. Here we only concern ourselves with fixed $k \geq 3$ and large $n$. Moore's bound says that $g(k,n) \le (1 \pm \oone)\cdot 2 \log_{k-1} (n)$. Although the argument is very simple, this remains our best asymptotic upper bound.

The study of high-girth graphs has a long history. Using a combinatorial argument, Erd\H{o}s and Sachs \cite{erdos1963regulare} proved in 1963 that $g(k,n) \ge \oneoone \log_{k-1} (n)$. Twenty years later, Biggs and Hoare \cite{biggs1983sextet} gave an algebraic construction of a family of cubic graphs later shown \cite{weiss1984girths} to have girth at least $(1 - \oone)\frac{4}{3} \log_2 (n)$. Then, for $k$ an odd prime plus one, Lubotzky, Phillips, and Sarnak \cite{lubotzky1988ramanujan} constructed their celebrated Ramanujan graphs, with girth $\oneoone \frac{4}{3} \log_{k-1} (n)$. As observed in \cite[Introduction]{hoory2002graphs}, this implies that $g(k,n)\ge (1-o(1)) c(k) \log_{k-1} (n)$, where $c(k) > 1$ for \textit{every} $k \geq 3$, and $\lim_{k\to\infty} c(k) = 4/3$. Cayley graphs attaining this bound were found by Dahan \cite{dahan2014regular}. Along the way, advances by Chiu \cite{chiu1992cubic}, Morgenstern \cite{morgenstern1994existence}, and Lazebnik, Ustimenko, and Woldar \cite{lazebnik1995new} broadened the range of degrees for which similar constructions are known. We further refer the reader to Biggs's survey \cite{biggs1998constructions} of the best known constructions for cubic graphs.

In contrast, and notwithstanding considerable research efforts, the Erd\H{o}s-Sachs bound remains the best asymptotic lower bound on $g(k,n)$ that is derived by combinatorial and probabilistic techniques. This is one of very few examples where explicit algebraic constructions beat the probabilistic method. We believe that the road to constructing high-girth graphs using such methods goes via better understanding of the \textit{large-scale geometry of graphs}. In our open problem section we mention several additional mysteries in this domain. 

Here we describe a \textit{random greedy algorithm} to construct regular high-girth graphs. In recent years, random greedy algorithms have become a powerful tool for constructing constrained combinatorial structures. Thus, Glock, K\"uhn, Lo, and Osthus \cite{glock2018conjecture}, and independently Bohman and Warnke \cite{bohman2018large}, used this method to prove the existence of approximate Steiner triple systems that are locally sparse. This methodology has also played prominent roles in the proofs by Keevash \cite{keevash2014existence} and Glock, K\"uhn, Lo, and Osthus \cite{glock2016existence} of the existence of combinatorial designs.

Random greedy algorithms have also been studied in their own right. For example, in the ``triangle-free'' graph process (e.g., \cite{erdos1995size, bohman2009triangle}), edges are randomly added to a graph one by one and subject to the constraint that no triangle is created. Similarly, various authors studied ``$H$-free'' processes for other fixed graphs $H$, including stars \cite{rucinski1992random} and cycles \cite{osthus2001random, bollobas2000constrained, picollelli2011final, warnke2014c_ell, picollelli2014final}. In another relevant paper Krivelevich, Kwan, Loh, and Sudakov \cite{krivelevich2018random} studied the process where edges are randomly added to a graph as long as the matching number remains below a fixed value which may depend on the number of vertices. This is indeed just a tiny sample of a rich and beautiful body of literature.

Here is a simple method to generate random $k$-regular graphs on $n$ vertices: Start with a Hamilton cycle, and repeatedly add perfect matchings uniformly at random until the desired degree is attained. Since the present paragraph is intended only as background, we do not go into detail, and do not dwell on how to avoid double edges. We consider here a sequential variant of this algorithm, which produces graphs of girth at least $g$. Let $G = (V,E)$ be a graph on $n$ vertices with all vertex degrees at most $k$ (in our main application, $G$ is a Hamilton cycle, and $k\ge 3$). Let $g \leq n$. Set $G_0 = G = (V,E_0)$. We obtain $G_{t+1} = (V,E_{t+1})$ from $G_t$ as follows: 
\begin{itemize}
	\item If $G_t$ is $k$-regular, set $G_{t+1} = G_t$.
	
	\item Otherwise:
	
	\begin{itemize}
		\item Let $d < k$ be the smallest vertex degree in $G_t$, and let $W_t$ be the set of \termdefine{unsaturated} vertices in $G_t$, i.e., those with degree $d$.
	
		\item We say that $u,v \in W_t$ is an \termdefine{available} pair of vertices if their distance in $G_t$ is at least $g-1$. Let $\mA_t$ be the set of available pairs, and let $H_t$ be the graph $(W_t, \mA_t)$.
		
		\item If $\mA_t = \emptyset$, set $G_{t+1} = G_t$.
		
		\item Otherwise, choose $e_{t+1} \in \mA_t$ uniformly at random, and set $E_{t+1} = E_t \cup \{e_{t+1}\}$.
	\end{itemize}
\end{itemize}

We call this the \termdefine{$(G,g,k)$-high-girth-process}. We say that the process \termdefine{saturates} if for some $t$, $G_t$ is $k$-regular. We note that in this case $\girth(G_t) \geq \min \left\{ g, \girth (G) \right\}$.

Our main result is that with proper choice of parameters, this algorithm yields high-girth regular graphs.

\begin{thm}\label{thm:main}
	Let $1> c > 0$, $k \geq 3$ an integer, and $n$ an even integer. Let $g = g(n) \leq c \log_{k-1} (n)$, and $G$ be a Hamilton cycle on $n$ vertices. Then, w.h.p.\footnote{We say that a sequence of events occurs \termdefine{with high probability} (\termdefine{w.h.p.}) if the probabilities of their occurrence tend to $1$.}, the $(G,g,k)$-high-girth-process saturates.
\end{thm}

A byproduct of the analysis of this algorithm is a lower bound on the number of high-girth regular graphs.

\begin{thm}\label{thm:enumeration}
	Let $1> c > 0$, $k \geq 3$ an integer, and $n$ an even integer. There are  at least $\left( \Omega (n) \right)^{kn/2}$ labeled $k$-regular graphs $G$ on $n$ vertices with $\girth(G) \ge c \log_{k-1} (n)$.
\end{thm}

\begin{rmk}
	We do not give Theorem \ref{thm:enumeration} in the best form known to us, since we believe this is in any rate far from the truth.
	
	We also mention that for $c < 1/2$, a remarkably accurate enumeration is given by McKay, Wormald, and Wysocka \cite[Corollary 2]{mckay2004short} who studied the distribution of the number of cycles in random regular graphs. However, they do not give a construction, and their method applies only when $c < 1/2$.
\end{rmk}

Theorem \ref{thm:enumeration} illustrates one advantage of probabilistic constructions over algebraic ones: While the latter achieve higher girth, they are sporadic and provide only a small supply of examples. Similarly, purely deterministic constructions (such as Erd\H{o}s and Sachs's) tend to be restrictive and difficult to analyze. In contrast, probabilistic techniques provide a viewpoint from which to study a very large family of high-girth graphs.

In comparison with other results in the literature, ours is the first probabilistic algorithm that constructs graphs with unbounded girth that are also regular. For constant $g$, Osthus and Taraz \cite[Corollary 4]{osthus2001random} determined (up to polylog factors) the final number of edges in the $\mH$-free process, where $\mH$ is the collection of all cycles shorter than $g$. Bayati, Montanari, and Saberi \cite{bayati2009generating} studied a similar sequential process which samples uniformly from the family of girth-$g$ graphs with $m$ edges, where $g$ is a constant and $m = O \left( n^{1 + \alpha(g)} \right)$, for some non-negative function $\alpha$. Chandran \cite{chandran2003high} considered a (deterministic) greedy algorithm to construct graphs with girth $(1+\oone) \log_{k} (n)$ and average degree $k$. However, none of these constructions produce regular graphs. Closer to the algebraic end of the spectrum, Gamburd, Hoory, Shahshahani, Shalev, and Vir\'ag \cite{gamburd2009girth} showed that for various families of groups, random $k$-regular Cayley graphs have unbounded girth that in some cases is as high as $(1-o(1)) \log_{k-1} (n)$.

The rest of this paper is organized as follows. The remainder of this section introduces some notations. In Section \ref{sec:main proof} we prove Theorem \ref{thm:main}, modulo two technical claims which are proved in Sections \ref{sec:proof of G' properties} and \ref{sec:random matching proof}. We prove Theorem \ref{thm:enumeration} in Section \ref{sec:enumeration}. We close with some remarks and open problems in Section \ref{sec:closing}.

\subsection{Notation}

The vertex and edge sets of a graph $G$ are denoted by $V(G)$, resp.\ $E(G)$. We write $e(G) = |E(G)|$. The neighbor set of vertex $v \in V(G)$ is denoted $\Gamma_G(v)$. The distance between $u,v \in V(G)$ is denoted $\delta_G(u,v)$. The graph of $G$ induced by $U \subseteq V(G)$ is denoted $G[U]$.

The set $\{1,2,\ldots,a\}$ is denoted by $[a]$. Also, $[a]_0:=\{0,1,2,\ldots,a\}$, and $\N_0:=\N \cup \{0\}$. For $x,y \in \R$, we write $x \pm y$ to indicate an unspecified number in the interval $[x-|y|,x+|y|]$.

\section{Constructing high-girth graphs: proof of Theorem \ref{thm:main}}\label{sec:main proof}

Let $G_0',G_1',\ldots$ be a $(G',g,k)$-high-girth-process, where $G'$ is a Hamilton cycle, and $c, k, n$ and $g \leq c\log_{k-1}(n)$ are as in the theorem. We argue by induction on $k$, starting with $k=3$. Now, suppose Theorem \ref{thm:main} holds for $k-1 \geq 3$. Then, since \mbox{$g \leq c \log_{k-1} (n) < c \log_{k-2} (n)$}, it follows by induction that w.h.p.\ $G_{(k-2)n/2}'$ is $(k-1)$-regular. It is thus sufficient to prove the following proposition (which covers both the base case and the inductive step).

\begin{prop}\label{prop: induction step}
	Let $G$ be a $(k-1)$-regular graph on $n$ vertices, with $n$ even and $k \geq 3$. Let $c < 1$ and let $g \leq c \log_{k-1} (n)$. Then, w.h.p., the $(G,g,k)$-high-girth-process saturates.
\end{prop}

Let $G_0,G_1,\ldots$ be a $(G,g,k)$-high-girth-process, and let $e_1,e_2,\ldots$ be the edges added to the graph at each step. Clearly, a necessary and sufficient condition for the process to saturate is that $|E_t| = (k-1)n/2 + t$ for every $0 \leq t \leq n/2$. We say that the process \termdefine{freezes at time $t$} if $t$ is the smallest integer such that $E_t = E_{t+1}$. We denote this time by $\tfr$ (so that the process saturates if and only if $\tfr = n/2$).

Our proof deals separately with two phases of the process. In Section \ref{ssec:early evolution} we show that in the first phase it holds with certainty that almost all vertices saturate, and $H_t$ is almost complete.

We begin Section \ref{ssec:latter evolution} by observing that in the special case where $c < 1/3$, the analysis in Section \ref{ssec:early evolution} suffices to conclude that the process saturates w.h.p. The remainder of Section \ref{ssec:latter evolution} is devoted to the more involved, ``nibbling''-based analysis of the second phase. We divide the remainder of the process into a bounded number of steps. We show that in each step, the number of unsaturated vertices is reduced by a polynomial factor, and that certain pseudorandomness conditions are preserved from step to step. We then argue that w.h.p.\ the graph obtained at the end of the penultimate step has a combinatorial property that implies the process saturates with certainty.

\subsection{The early evolution of the process}\label{ssec:early evolution}

Let $0<\varepsilon < 1-c$, and let
\[
T = \frac{1}{2} \left(n - n^{c+\varepsilon} \right).
\]

The following observation follows from the Moore bound.

\begin{obs}\label{obs:bfs bound}
	Let $H$ be a graph with maximal degree at most $k$, and let $\ell \in \N$. For every $v \in V(H)$ there are at most $k \cdot (k-1)^{\ell-1}$ vertices at distance $\ell$ from $v$, and at most $2k \cdot (k-1)^\ell$ vertices at distance at most $\ell$ from $v$.
\end{obs}

We next use this observation to show that for $n$ sufficiently large $\tfr \geq T$ with certainty, and for every $t \leq T$, the graph $H_t$ is nearly complete.

\begin{lem}\label{lem:initial phase}
	 There exists an integer $n_0=n_0(c,\varepsilon)$ such that for all $n \geq n_0$ and every $t \leq T$ it holds with certainty that:
	\begin{enumerate}
		\item All vertex degrees in $H_t=(W_t, \mA_t)$ are at least $|W_t| - O(n^c)$.
		\item\label{itm:unsat edges lower bound} $|\mA_t| = \frac{1}{2} |W_t|^2 \left( 1 - O(n^c / |W_t|) \right)$.
		\item\label{itm:unsat lower bound} $|W_t| = n - 2t$, and hence
		\item\label{itm:tfr lower bound} $\tfr \geq T$.

	\end{enumerate}
\end{lem}

\begin{proof}
	The two vertices of every edge in $E_t \setminus E_0$ have degree $k$, and every vertex of degree $k$ is in exactly one edge from $E_t \setminus E_0$. Therefore: $|W_t| = n - 2|E_t \setminus E_0| \geq n - 2t$, with equality if and only if $t \leq \tfr$.
	
	Let $v \in W_t$. By Observation \ref{obs:bfs bound}, there are $O(n^c)$ vertices in $G_t$ with distance at most $g-2$ to $v$. In $H_t$, $v$ is adjacent to all other vertices in $W_t$. Therefore $d_{H_t}(v) \geq |W_t| - O(n^c)$, as claimed. Hence,
	\[
	|\mA_t| = \frac{1}{2} \sum_{v \in W_t} d_{H_t}(v) = \left( 1 - O(n^c / |W_t|) \right) \frac{1}{2} |W_t|^2,
	\]
	as desired.
	
	Finally, $\tfr \geq T$ as long as $\mA_T \neq \emptyset$. As observed:
	\[
	|W_T| \geq n - 2T = n^{c+\varepsilon}.
	\]
	Hence, by \ref{itm:unsat edges lower bound}:
	\[
	|\mA_T| = \frac{1}{2} |W_T|^2 \left( 1 - O \left( \frac{n^c}{|W_T|} \right) \right) = \Omega \left( n^{2(c+\varepsilon)} \right).
	\]
	Thus, if $n$ is large enough, then $\mA_T$ is nonempty with certainty, implying \ref{itm:unsat lower bound}, \ref{itm:tfr lower bound}.
\end{proof}

The set $\mB_t$ of \termdefine{forbidden edges} is comprised of those pairs $u,v \in W_t$ with $uv \notin \mA_t$. We will show that for $t \geq T$ w.h.p.\ the number of forbidden edges in $G_t$ does not exceed the bound given by the following heuristic argument. Let $u \in W_t$. By Observation \ref{obs:bfs bound}, at most $n^c$ vertices $v \in V$ satisfy $\delta_{G_t} (u,v) \leq g-2$. So, if $v$ is chosen randomly from $V$, then ${\Prob [v \in W_t]} = |W_t|/n$ and $\Prob [\delta_{G_t} (u,v) \leq g-2] \leq n^c/n$. Had these events been independent, we expect there to be at most $|W_t|^2 n^c/n$ pairs $u,v \in W_t$ with $\delta_H(u,v) \leq g-2$. Hence, when $|W_t| \ll \sqrt{n/n^c}$, we expect that $\mB_t = \emptyset$. We now show that the latter condition implies that the process saturates with certainty.

\begin{definition}
	Let $G = (V,E)$ be a graph with all degrees either $k-1$ or $k$. We say that $G$ is \termdefine{safe} if every two vertices of degree $k-1$ are at distance $\geq g-1$.
\end{definition}

Clearly $G_t$ is safe if and only if $\mB_t = \emptyset$.

\begin{lem}\label{lem:safe graph}
	If for some $t$, $G_t$ is safe, then the process saturates with certainty.
\end{lem}

\begin{proof}
	We first observe that if $G_t$ is safe then $H_t$ is the complete graph on $W_t$. Thus, it is enough to show that if $t < n/2$, then $\mA_t \neq \emptyset$ and $G_{t+1}$ is also safe. Suppose, for a contradiction, that $e_{t+1} = uv$ for some $u,v \in W_t$, and that $G_{t+1}$ is not safe. Namely, there exist two vertices $a,b \in W_{t+1}$ such that $\delta_{G_{t+1}} (a,b) \leq g-2$. Let $P$ be a shortest $ab$-path in $G_{t+1}$. By assumption, its length is at most $g-2$. But $G_t$ is safe, whence $\delta_{G_{t}} (a,b) \geq g-1$, so that necessarily $uv \in P$. It follows that in $G_t$ there is a path of length $\le g-2$ from one of the vertices $a,b$ to one of $u,v$ contrary to the assumption that $G_t$ is safe.
\end{proof}

Here is the main technical ingredient in the analysis of the first $T$ steps in the process. An edge $uv$ is a \termdefine{chord} if it is not in the initial graph $G$. E.g., all edges chosen by the process are chords.

\begin{clm}\label{clm:subset containment bound with times}
	Let $a,b \leq \log^2 (n)$. Let $s_1,s_2,\ldots,s_a$ be distinct chords and let $U\subseteq V$ be a set of $b$ vertices. Let $0 \leq t_1,t_2,\ldots,t_a < T$. Let $A$ be the event that for every $1 \leq i \leq a$, the process chooses chord $s_i$ at step $t_i$ (i.e., $e_{t_i} = s_i$), and that $U \subseteq W_T$. Then
	\[
	\Prob [A] \leq \oneoone \left( \frac{2}{n^2} \right)^a \left( 1 - \frac{2T}{n} \right)^b.
	\]
\end{clm}

It is easy to see where this expression comes from. Since $|W_T| = n-2T$, it is plausible that ${\Prob [v \in W_T]} \approx 1 - 2T/n$ for every $v \in V$. Also, if edges are chosen uniformly at random, ignoring the degree and girth constraints, then the probability of the event $e_{t} = s$ is $\oneoone 2n^{-2}$. The bound on $\Prob [A]$ says that the constraints can only reduce this probability. This heuristic will be justified by Lemma \ref{lem:initial phase}: Throughout the first $T$ steps of the process, the graph of available edges $H_t$ is nearly complete. Therefore, in each of the first $T$ steps, both the number of available edges and the number of available edges incident to $U$ are very close to what these values would be in the unconstrained graph process. As a consequence, the two processes exhibit similar behavior.

\begin{proof}
	Note that there is no loss in assuming that
	\begin{itemize}
		\item $s_1,\ldots,s_a$ form a matching,
		\item $t_1,\ldots,t_a$ are all distinct, and
		\item $U$ is disjoint from the vertices in $s_1,\ldots,s_a$,
	\end{itemize}
for otherwise $\Prob[A] = 0$ and the conclusion follows trivially. 
	
	The sequential nature of the process suggests that we express $A$ as an intersection of events $B_0, B_1, \ldots, B_T$, where $B_t$ depends only on the chord selected at step $t$. For $0 \leq t < T$, let $S_t = \{ s_i : t_i > t \}$ be the set of chords that are to be chosen after step $t$. Let $U_t = {U \cup \{ u : \exists s \in S_t, u \in s \}}$. The definition of $B_t$ depends on whether or not $t = t_i$ for some $i$. If so, we let $B_t$ be the event that chord $s_i$ is selected at step $t$. Otherwise, it is the event that we select at step $t$ a chord disjoint from $U_t$. Clearly,
	\[
	A = B_0 \cap B_1 \cap \ldots \cap B_T.
	\]
	Therefore:
	\begin{equation}\label{eq:sequential probability}
	\Prob [A] = \Prob [B_0] \times \Prob[B_1 \given B_0] \times \Prob [B_2 \given B_1 \cap B_0] \times \ldots \times \Prob [B_{T-1} \given B_0 \cap \ldots \cap B_{T-2}].
	\end{equation}
	By Lemma \ref{lem:initial phase}, for every $t < T$, we have
	\[
	|\mA_t| = \left( 1 \pm O \left( \frac{n^c}{n-2t} \right) \right) \frac{(n-2t)^2}{2}.
	\]
	Therefore, for every $i \in [a]$, it holds that
	\begin{equation*}
	\Prob [B_{t_i} \given B_{t_i-1} \cap \ldots \cap B_0] \leq \left( 1 + O \left( \frac{n^c}{n-2t_i} \right) \right) \frac{2}{(n-2t_i)^2}.
	\end{equation*}
	It will be useful to note also that
	\begin{equation}\label{eq:t in a}
	\begin{split}
	\prod_{i=1}^a & \Prob [B_{t_i} \given B_{t_i-1} \cap \ldots \cap B_0] \leq \prod_{i=1}^a \left( 1 + O \left( \frac{n^c}{n-2t_i} \right) \right) \frac{2}{(n-2t_i)^2}\\
	& \leq \oneoone \prod_{i=1}^a \left( 1 + O \left( \frac{n^c}{n-2t_i} \right) \right) \frac{2}{(n-2t_i)^2} \left( 1 - \frac{2|U_{t_i}|}{n-2t_i} \right).
	\end{split}
	\end{equation}
	
	Consider next the case $t \notin \{t_1,\ldots,t_a\}$. The event $B_0 \cap \ldots \cap B_{t-1}$ implies that $U_t \subseteq W_t$, so by Lemma \ref{lem:initial phase}, $U_t$ intersects at least
	\[
	\left( 1 \pm O \left( \frac{n^c}{n-2t} \right) \right) |U_t| (n-2t) \pm \binom{|U_t|}{2} = \left( 1 \pm O \left( \frac{n^c}{n-2t} \right) \right) |U_t| (n-2t)
	\]
	chords in $\mA_t$. Thus:
	\begin{equation}\label{eq:t notin a}
	\begin{split}
	\Prob [B_{t} \given B_{t-1} \cap \ldots \cap B_0]
	& \leq 1 - \left( 1 \pm O \left( \frac{n^c}{n-2t} \right) \right)\frac{|U_t| (n-2t)}{|\mA_t|}\\
	& \leq 1 - \left( 1 \pm O \left( \frac{n^c}{n-2t} \right) \right) \frac{2|U_t|}{n-2t}.
	\end{split}
	\end{equation}
	Therefore, by \eqref{eq:sequential probability}, \eqref{eq:t in a}, and \eqref{eq:t notin a}:
	\begin{equation}\label{eq:Prob A bound}
	\begin{aligned}
	\Prob[A] & \leq \oneoone \left( \prod_{i=1}^a \frac{2}{(n-2t_i)^2} \right) \left( \prod_{t=0}^{T} \left( 1 - \left( 1 \pm O \left( \frac{n^c}{n-2t} \right) \right) \frac{2|U_t|}{n-2t} \right) \right) \\
	& \leq \oneoone \left( \prod_{i=1}^a \frac{2}{(n-2t_i)^2} \right) \exp \left( - \sum_{t=0}^T \left( 1 \pm O \left( \frac{n^c}{n-2t} \right) \right) \frac{2|U_t|}{n-2t} \right) \\
	& \leq \oneoone \left( \prod_{i=1}^a \frac{2}{(n-2t_i)^2} \right) \exp \left( - \sum_{t=0}^T \frac{2|U_t|}{n-2t} \right).
	\end{aligned}
	\end{equation}
	We now estimate the sum in the exponent. By definition of $U_t$, we have:
	\[
	|U_t| = |U| + 2 \left| S_t \right| = b + 2 \left| S_t \right|.
	\]
	It follows that:
	\[
	\sum_{t=0}^T \frac{2|U_t|}{n-2t} = b \sum_{t=0}^T \frac{2}{n-2t} + \sum_{i=1}^a \sum_{t=0}^{t_i} \frac{4}{n-2t}
	= b \sum_{t=0}^T \frac{1}{n/2-t} + 2 \sum_{i=1}^a \sum_{t=0}^{t_i} \frac{1}{n/2-t}.
	\]
	We recall that $\sum_{k=\ell}^L 1/k \geq \log \left( L/\ell \right)$ holds whenever $\ell \leq L$. Therefore:
	\begin{align*}
	\sum_{t=0}^T \frac{2|U_t|}{n-2t} & \geq b \log \left( \frac{n}{n-2T} \right) + 2\sum_{i=1}^{a} \log \left( \frac{n}{n-2t_i} \right).
	\end{align*}
	Plugging this into \eqref{eq:Prob A bound}, we obtain:
	\begin{align*}
	\Prob [A] & \leq \oneoone \left( \prod_{i=1}^a \frac{2}{(n-2t_i)^2} \right) \exp \left( b \log \left( \frac{n-2T}{n} \right) + 2\sum_{i=1}^{a} \log \left( \frac{n-2t_i}{n} \right) \right) \\
	& \leq \oneoone \left( \prod_{i=1}^a \left( \frac{2}{(n-2t_i)^2} \left( \frac{n-2t_i}{n} \right)^2 \right) \right) \left( 1 - \frac{2T}{n} \right)^b \\
	& \leq \oneoone \left( \frac{2}{n^2} \right)^a \left( 1 - \frac{2T}{n} \right)^b,
	\end{align*}
	as claimed.
\end{proof}

Claim \ref{clm:subset containment bound with times} helps us bound the probability that $E_T$ contains a given set of edges:

\begin{lem}\label{lem:subset containment bound}
	Let $S$ be a set of $a \leq \log^2 n$ chords, and let $u,v \in V$ be distinct vertices. The probability that $S \subseteq E_T$ and that $u,v \in W_T$ is $O(n^{-(a+2(1-c-\varepsilon))})$.
\end{lem}

\begin{proof}
	Let $s_1,\ldots,s_a$ be an ordering of the chords in $S$. We employ a union bound over all times $t_1,\ldots,t_a$ such that for every $i$, the chord chosen at step $t_i$ is $s_i$. By Claim \ref{clm:subset containment bound with times} the probability that $S \subseteq E_T$ and both $u$ and $v$ have degree $2$ in $G_T$ is at most
	\begin{align*}
	\sum_{0 \leq t_1,\ldots,t_a \leq T} \oneoone \left( \frac{2}{n^2} \right)^a \left( 1 - \frac{2T}{n} \right)^2
	& = O \left( T^a \left( \frac{2}{n^2} \right)^a \left(\frac{n^{c+\varepsilon}}{n}\right)^2 \right)\\
	& = O \left( \frac{1}{n^{a+2(1-c-\varepsilon)}} \right),
	\end{align*}
	as desired.
\end{proof}

Lemma \ref{lem:subset containment bound} allows us to bound the number of edges in $\mB_T$. If $uv \in \mB_T$, then $G_T$ contains a path $P$ from $u$ to $v$ of length $\leq g-2$ such that:
\begin{itemize}
	\item No two consecutive edges in $P$ are chords.
	
	\item The first and the last edge in $P$ are not chords.
\end{itemize}
A length-$\ell$ path in $K_n$ satisfying these conditions is said to be \termdefine{$\ell$-threatening}.

We also introduce several random variables that will allow us to bound the size of $\mB_t$ throughout the process. For $\ell \leq g-2$, we denote by $P_\ell(G_t)$ the number of pairs $u,v \in W_t$ such that $\delta_{G_t}(u,v) = \ell$. Additionally, for $v \in V$, we denote by $P_\ell(G_t,v)$ the number of vertices $u \in W_t$ such that $\delta_{G_t} (u,v) = \ell$.

\begin{lem}\label{lem:path bound}
	Let $\ell \in \N$ and let $a \in \N_0$ such that $2a+1 \leq \ell$. Then $K_n$ contains fewer than $n^{a+1} (k-1)^\ell$ $\ell$-threatening paths containing $a$ chords.
\end{lem}

\begin{proof}
	We prove the lemma by considering the number of $\ell$-threatening path with $a$ chords. There are $n$ choices for the initial vertex. Since the first edge is not a chord, it must be one of the $k-1$ edges in $G$ that are incident to the initial vertex. Then, for each subsequent edge, there are two possibilities:
	\begin{itemize}
		\item If the previous edge was a chord, the next edge must be one of the $k-1$ edges in $G$ incident to the current vertex.
		
		\item Otherwise, the next edge is either one of the $k-2$ non-backtracking edges incident to the current vertex, or else it is a chord. In this case there are $n-k<n$ choices for the chord.
	\end{itemize}
	Put differently, at each step, there are $k-1$ basic choices: Either the $k-1$ edges incident to the current vertex, or else the $k-2$ non-backtracking edges incident to the current vertex together with the choice ``chord''. If the choice is ``chord'', there are (fewer than) $n$ further choices of the specific chord. Since we are considering length-$\ell$ paths with $a$ chords, the total number of choices is at most $n^{a+1}(k-1)^\ell$, as desired.
\end{proof}

We next give upper bounds on $P_\ell(G_T)$ and $P_\ell(G_T, v)$ for $\ell \leq g-2$ and $v \in W_T$.

\begin{lem}\label{lem:G_T threatening paths}
	The following hold w.h.p.:
	\begin{enumerate}
		\item For every $\ell \leq g-2$, $P_\ell (G_T) \leq |W_T|^2 \frac{(k-1)^\ell}{n} \log^3 (n)$.
		
		\item\label{itm:G_T vertex paths} For every $\ell \leq g-2$ and every $v \in W_T$, $P_\ell (G_T,v) \leq (k-1)^\ell$.
	\end{enumerate}
\end{lem}

\begin{proof}
	We calculate the expected number of length-$\ell$ paths between vertices in $W_T$. By Lemma \ref{lem:path bound}, there are at most $n^{a+1}(k-1)^\ell$ $\ell$-threatening paths in $K_n$ with $a$ chords. By Lemma \ref{lem:subset containment bound}, for each such path, the probability that it is contained in $E(G_T)$ and that its two endpoints are in $W_T$ is at most $O \left( n^{-a - 2(1-c-\varepsilon)} \right)$. Therefore:
	\begin{align*}
	\E \left[ P_\ell(G_T) \right] & = O \left( \sum_{a=0}^{(\ell-1)/2} \frac{n^{a+1}(k-1)^\ell}{n^{a + 2(1-c-\varepsilon)}} \right)
	= O \left( \frac{n^{2(c+\varepsilon)}}{n^2} \sum_{a =0}^{(\ell-1)/2} \frac{n^{a+1} (k-1)^\ell}{n^a} \right)\\
	& = O \left( \frac{|W_T|^2}{n} \sum_{a=1}^{(\ell-1)/2} (k-1)^\ell \right)
	= O \left( |W_T|^2 \frac{(k-1)^\ell}{n}\log (n) \right).
	\end{align*}
	Therefore, by Markov's inequality, for every $\ell$, it holds that
	\[
	\Prob \left[ P_\ell(G_T) \geq |W_T|^2 \frac{(k-1)^\ell}{n} \log^3 (n) \right] = O \left( \frac{1}{\log^2 (n)} \right).
	\]
	Applying a union bound to the $O(\log (n))$ random variables $P_1(G_T),\ldots,P_{g-2}(G_T)$, we conclude that w.h.p., for every $1 \leq \ell \leq g-2$, it holds that
	\[
	P_\ell(G_T) \leq |W_T|^2 \frac{(k-1)^\ell}{n} \log^3 (n),
	\]
	as desired.
	
	Part \ref{itm:G_T vertex paths} follows from Moore's bound (Observation \ref{obs:bfs bound}). For every $v \in W_T$ there are at most $(k-1)^\ell$ vertices $u \in V$ with $\delta_{G_T}(u,v) = \ell$. In particular, there are at most $(k-1)^\ell$ such vertices in $W_T$.
\end{proof}

\subsection{The latter evolution of the process}\label{ssec:latter evolution}

Let
\[
\varepsilon = \frac{c(1-c)}{3},\quad T = \frac{1}{2} (n-n^{c+\varepsilon}), \quad T_{safe} = \frac{1}{2} (n-n^\varepsilon).
\]

Our plan is to show that w.h.p.\ $G_{T_{safe}}$ is safe.

Our analysis of the first $T$ steps of the process used rather crude tools: Moore's bound, and a first-moment calculation. Analyzing the remaining steps of the process is more involved. However, this more involved analysis is not necessary if already $G_T$ is safe. This is indeed the case w.h.p.\ if $c < 1/3$, as we now show.

\begin{proof}[Proof of Theorem \ref{thm:main} when $c < 1/3$]
	Suppose $c < 1/3$. By Lemma \ref{lem:G_T threatening paths}, w.h.p., for every $1 \leq \ell \leq g-2$, we have:
	\begin{align*}
	P_\ell(G_T) & \leq |W_T|^2 \frac{(k-1)^\ell}{n} \log^3 (n) \leq n^{2(c+\varepsilon)}\frac{(k-1)^\ell}{n} \log^3(n)\\
	& \leq n^{2c+2\varepsilon} \frac{n^c}{n} \log^3(n) \leq n^{3c + 2\varepsilon - 1} \log^3(n) = \oone.
	\end{align*}
	Therefore, for every $1 \leq \ell \leq g-2$, it holds that $P_\ell(G_T) = 0$. In other words $G_T$ is safe, as claimed.
\end{proof}

We return to our main narrative with $1 > c \geq 1/3$. We define:
\[
\beta = \betadef, \quad \alpha = \alphadef.
\]
We have chosen these particular constants for concreteness; all we need is that $\beta$ is sufficiently smaller than $\varepsilon$ and that $\alpha$ is sufficiently smaller than $\beta$.
By Lemmas \ref{lem:initial phase} and \ref{lem:G_T threatening paths} the following hold w.h.p.\ (in fact, \ref{itm:G_T certainty first} - \ref{itm:G_T certainty last} hold with certainty):
\begin{enumerate}
	\item\label{itm:G_T certainty first} $\tfr \geq T$.
	
	\item $|W_T| = n-2T = n^{c+\varepsilon}$.
	
	\item\label{itm:G_T certainty last} For every $v \in W_T$ and $\ell \leq g-2$, there holds $P_{\ell}(G_T, v) \leq (k-1)^\ell = \frac{(k-1)^\ell}{n^{c+\varepsilon}}|W_T|$.
	
	\item For every $\ell \leq g-2$, there holds $P_\ell (G_T) \leq \frac{|W_T|^2 (k-1)^\ell}{n} \log^3(n)$. In particular, $|\mB_T| \leq \frac{|W_T|^2 n^c}{n} \log^4(n)$.
\end{enumerate}

These are pseudorandom properties of $G_T$: The number of pairs of vertices in $W_T$ at distance $\ell \leq g-2$ does not exceed its expectation by more than a polylog factor, and no vertex in $W_T$ is close to too many other vertices in $W_T$. As we show, if these pseudorandomness conditions hold at step $T_{safe} \geq t \geq T$, then w.h.p.\ they persist until step $t' = (n - (n-2t)n^{-\alpha})/2$. In particular, between times $t$ and $t'$ the number of unsaturated vertices gets multiplied by $n^{-\alpha}$, while the number of forbidden pairs is multiplied by $n^{-2\alpha}$ (ignoring polylog factors). If we repeat this process $(c+\varepsilon) / \alpha = O(1)$ times, then w.h.p.\ no forbidden pairs remain, i.e., the graph is safe. By Lemma \ref{lem:safe graph}, this implies that the process saturates.

We make this precise in the next lemma, which is the heart of our proof. It is useful to introduce the following real function $L$:
\[
L(\ell, t) = \max \left\{ 1, \frac{(k-1)^\ell (n-2t)}{n^{c+\varepsilon}} \right\}.
\]

We now formally define the pseudorandomness properties.

\begin{definition}\label{def:path bounded}
	For $C > 0$ we say that $G_t$ is \termdefine{$C$-path-bounded} if:
	\begin{enumerate}
		\item\label{itm:local path count} $P_\ell(G_t,v) \leq L(\ell,t) \log^C (n)$ for every $v \in W_t$ and every $\ell \leq g-2$.
		
		\item\label{itm:global path bound} $P_\ell(G_t) \leq \frac{|W_t|^2 (k-1)^\ell}{n} \log^C (n)$ for every $\ell \leq g-2$.
	\end{enumerate}
\end{definition}

We remark that for every $C$ there exists some $n_0 = n_0(C)$ such that if $n \geq n_0$ and $G_t$ is $C$-path-bounded for some $T \geq t \geq T_{safe}$, then $|W_t| = n-2t$. This is because \ref{itm:global path bound} implies that $|\mB_t| = O \left( |W_t|^2 n^c \log^C (n)/n \right) = o (|W_t|^2)$. Therefore, if $n$ is large enough then $\mA_t$ is not empty, meaning $\tfr \geq t$, and $|W_t| = n-2t$.

\begin{lem}\label{lem:key matching lemma}
	There is a function $D=D(C)$ such that for every $T_{safe} \geq t \geq T$, if $G_t$ is $C$-path-bounded then, for $t' = (n - {(n-2t)n^{-\alpha}})/2$, w.h.p.\ $G_{t'}$ is $D$-path-bounded.
\end{lem}

Lemma \ref{lem:key matching lemma}, yields Proposition \ref{prop: induction step}, and hence Theorem \ref{thm:main}.

\begin{proof}[Proof of Proposition \ref{prop: induction step}]
	Define the sequence of integers $t_0 = T$, and for $i \geq 0$, $t_{i+1} := (n - (n-2t_i)n^{-\alpha})/2$. Let $m \coloneqq (c+\varepsilon) / \alpha$. Clearly $m = O(1)$. We observe that for every $i \leq m$, there holds $(n-2t_i)^2 = (n-2t)^2n^{-2\alpha i}$. In particular, $(n-2t_m)^2 n^{c-1} = n^{-\Omega(1)}$.
	
	As observed above, $G_{t_0} = G_T$ is $C_0 \coloneqq 3$-path-bounded. Therefore, by Lemma \ref{lem:key matching lemma}, w.h.p.\ $G_{t_1}$ is $C_1 \coloneqq D(C_0)$-path-bounded. Proceeding by induction, we conclude that w.h.p.\ $G_{t_m}$ is $C_m$-path-bounded, with $C_m = O(1)$. In particular, $|W_{t_m}| = n-2t_m$ and $|\mB_{t_m}| \leq |W_{t_m}|^2 n^{c-1} \log^{C_m} (n) = (n-2t_m)^2 n^{c-1} = \oone$. Namely, $\mB_{t_m} = \emptyset$, i.e., $G_{t_m}$ is safe, and by Lemma \ref{lem:safe graph} the process saturates.
\end{proof}

We turn to prove Lemma \ref{lem:key matching lemma}. We wish to analyze the $t'-t = |W_t| (1 - n^{-\alpha})/2$ steps of the process $G_t,G_{t+1},\ldots,G_{t'}$. We do so by viewing the process as taking place in two stages: Recall that $H_t = (W_t, \mA_t)$ is the graph of available edges. In the first stage, we take a random subgraph $H \subseteq H_t$, where $V(H) = W_t$ and every edge in $E(H_t) = \mA_t$ is included in $E(H)$ with probability $p \coloneqq n^\beta / |W_t|$ (with all choices independent). We also define the graph $G' = (V, E(G_t) \cup E(H))$. In the second stage, we run the high-girth process beginning from $G_t$, \textit{but using only the edges in $E(H)$}. This is similar to the ``honest nibble'' used by Grable \cite{grable1997random} to analyze random greedy triangle packing.

The advantage of this approach is that we can use standard tools for analyzing random binomial graphs to obtain properties of $H$ and $G'$. Significantly, there are very few ways that adding a matching from $H$ to $G_t$ might create a cycle shorter than $g$. This implies that the high-girth process run ``inside'' $G'$ behaves similarly to the random greedy matching algorithm in $H$. Finally, $H$ is sufficiently regular that the random greedy matching algorithm in $H$ succeeds, with high probability, in matching all but at most $|W_t|n^{-\alpha}$ vertices in $W_t$.

Formally, we define the process $G_t',G_{t+1}',\ldots$ as follows. To start, $G_t' = G_t$. Given $G_i'$, if there exist edges $uv \in E(H)$ such that $d_{G_i'}(u) = d_{G_i'} (v) = k-1$, and $\delta_{G_i'}(u,v) \geq g-1$, then choose such an edge $e$ uniformly at random and set $G_{i+1}' = \left( V(G), E(G_i') \cup \{e\} \right)$. If no such edges exist, set $G_{i+1}' = G_i'$. Let $T_{freeze}'$ be the smallest integer $i$ such that $G_i' = G_{i+1}'$.

We couple $G_t,G_{t+1},\ldots,G_{t'}$ and $G_t',G_{t+1}',\ldots,G_{t'}'$ by setting $G_i = G_i'$ for every $t < i \leq T_{freeze}'$. For $i > T_{freeze}'$, we obtain $G_{i+1}$ from $G_i$ independently of $G_{i+1}'$.

Clearly, for every $i \geq t$, it holds that $E(G_{i}') \subseteq E(G')$. We will show that w.h.p.\ $T_{freeze}' \geq t'$, and hence $G_{t'} = G_{t'}'$. It will then follow from the analysis of the process $G_0',\ldots,G_{t'}'$ that $G_{t'}$ is $D$-path-bounded for an appropriate $D$.

In order to track the process $G_t',G_{t+1}',\ldots$, we first identify the pairs of vertices $u,v \in W_t$ that \textit{might} have distance $\leq g-2$ in $G_{t'}'$. We observe that since $G_{t'}' \subseteq G'$, if $\delta_{G_{t'}'} (u,v) = \ell \leq g-2$ then there exists a sequence of vertices $w_0,w_1,\ldots,w_{2m-1}$ in $W_t$ such that:
\begin{itemize}
	\item $w_0 = u$ and $w_{2m-1} = v$.
	
	\item For every $1 \leq i \leq m-1$, it holds that $w_{2i-1}w_{2i} \in E(H)$. 
	
	\item $m-1 + \delta_{G_t}(w_0,w_1) + \delta_{G_t}(w_2,w_3) + \ldots + \delta_{G_t}(w_{2m-2}, w_{2m-1}) = \ell$.
\end{itemize}
This is similar to the observation in Section \ref{ssec:early evolution} that $uv \in \mB_T$ only if the chords from a threatening path in $K_n$ were chosen in the first $T$ steps of the process. We say that a pair of vertices $u,v \in W_t$ is \termdefine{$\ell$-threatened} if there exists a sequence of vertices satisfying these conditions. In this case, we say that the sequence $w_0,\ldots,w_{2m-1}$ \termdefine{witnesses} this fact.

We remark that the notion of an $\ell$-threatened pair of vertices is similar, but distinct from, the notion of an $\ell$-threatening path. Indeed, the latter refers to the specific \textit{path}, while the former to the endpoints. Furthermore, $\ell$-threatening paths are allowed to use any chord from $K_n$, whereas if $w_0,\ldots,w_{2m-1}$ is a witness that $w_0,w_{2m-1}$ are $\ell$-threatened then the edges $w_1w_2,w_3w_4,\ldots,w_{2m-2}w_{2m-1}$ must be in the random graph $H$.

For $1 \leq \ell \leq g-2$, let $T_\ell$ denote the number of $\ell$-threatened pairs in $W_t$. For a vertex $v \in W_t$, let $T_\ell(v)$ denote the number of $\ell$-threatened pairs that include $v$.

In the next claim we establish pseudorandom properties of $H$ and $G'$. These follow from standard techniques in the analysis of functions of independent random variables. In order not to interrupt the narrative, we defer the proof to Section \ref{sec:proof of G' properties}.

\begin{clm}\label{clm:G' properties}
	There exists a function $Q = Q(C)$ such that for every $T_{safe} \geq t \geq T$ if $G_t$ is $C$-path-bounded then, with $H$ and $G'$ defined as above, the following hold w.h.p.:
	\begin{enumerate}
		\item\label{itm:G' degrees} For every $v \in W_t$, $d_H (v) = \left( 1 \pm n^{-0.4\beta} \right) n^{\beta}$.
		
		\item\label{itm:path count G'} For every $\ell \leq g-2$ it holds that $T_\ell \leq |W_t|^2 \frac{(k-1)^\ell}{n} \log^{Q}(n)$.
		
		\item\label{itm:vertex path count G'} For every $v \in W_t$ and every $\ell \leq g-2$ it holds that $T_\ell (v) \leq L(\ell, t) \log^{Q} (n)$.
		
		\item\label{itm:vertex forbidden count G'} For every $v \in W_t$ there are at most $\log (n)$ vertices $u \in W_t$ such that $uv \in E(H)$ and $u,v$ are $\ell$-threatened for some $\ell \leq g-2$.
		
		\item\label{itm:G' discrepancy} For every $S \subseteq W_t$ such that $|S| \leq |W_t| / n^{\varepsilon/2}$, it holds that $e \left( H[S] \right) \leq |S|n^{0.9\beta}$.
	\end{enumerate}
\end{clm}

We turn to establish properties of the process $G_t',G_{t+1}',\ldots$. For $s \in \N_0$, let $U_s$ denote the set of degree-$(k-1)$ vertices in $G_{t+s}'$. Our intuition is that for every $s \leq t' - t$, $U_s$ resembles a random subset of $W_t$ with density $1 - 2s/|W_t|$. This implies, first, that $T_{freeze}' \geq t'$, and therefore $G_{t'} = G_{t'}'$. Second, this means that for $\ell \leq g-2$, the number of $\ell$-threatened pairs in $G'$ in which both vertices remain unsaturated in $G_{t'}$ is approximately $\left( |U_{t'}| / |W_t| \right)^2 T_\ell = n^{-2\alpha} T_\ell$. A similar statement holds for $\ell$-threatened pairs that contain a specific vertex. We conclude that $G_{t'}$ is, w.h.p., $D$-path-bounded for an appropriate $D$.

\begin{clm}\label{clm:RGMA in G'}
	There exists a function $D = D(Q)$ such that if $H$ and $G'$ satisfy conclusions \ref{itm:G' degrees}-\ref{itm:G' discrepancy} of Claim \ref{clm:G' properties} then, for $t' = (n - (n-2t)n^{-\alpha})/2$, w.h.p.\ $G_{t'}$ is $D$-path-bounded.
\end{clm}

While Claim \ref{clm:RGMA in G'} follows from methods developed to study random greedy hypergraph matching (e.g., \cite{bennett2012natural} and \cite[Section 4]{kwan2016almost}), it does not seem to follow directly from any explicit result in the literature. For completeness' sake we prove it in Section \ref{sec:random matching proof}.

We are now ready to prove Lemma \ref{lem:key matching lemma}.

\begin{proof}[Proof of Lemma \ref{lem:key matching lemma}]
	Suppose that for $T_{safe} \geq t \geq T$ and $C \in \R$, $G_t$ is $C$-path-bounded. Let $H$ and $G'$ be as above and let $t' = (n - (n-2t)n^{-\alpha})/2$ be as in the statement of Lemma \ref{lem:key matching lemma}. Then, w.h.p.\ $H$ and $G'$ satisfy the conclusions of Claim \ref{clm:G' properties} for some $Q = Q(C)$. Consequently, by Claim \ref{clm:RGMA in G'}, w.h.p.\ $G_{t'}$ is $D$-path-bounded for some $D = D(Q)$.
\end{proof}

\section{Proof of Claim \ref{clm:G' properties}}\label{sec:proof of G' properties}

Claim \ref{clm:G' properties} follows from standard arguments in the analysis of functions of independent random variables. We recall the following version of Chernoff's inequality.

\begin{thm}[Chernoff's Inequality]
	Let $X_1,X_2,\ldots,X_N$ be independent Bernoulli random variables, let $X = \sum_{i=1}^N X_i$, and let $\delta \in (0,1)$. Then:
	\[
	\Prob \left[ \left| X - \E [X] \right| \geq \delta \E [X] \right] \leq 2 \exp \left( - \frac{1}{3} \delta^2 \E [X] \right).
	\]
\end{thm}

We use a theorem of Kim and Vu to show concentration of multivariate polynomials. The setup is this: Let $Y = (V(Y), E(Y))$ be a hypergraph where the largest hyperedge size is $K = O(1)$ and $|V(Y)| = \omegaone$. Let $\{X_v\}_{v \in V(Y)}$ be a collection of independent Bernoulli random variables. Consider the random variable:
\[
X = \sum_{e \in E(Y)} \prod_{v \in e} X_v.
\]

For every $S \subseteq V(Y)$, define the random variable:
\[
X_S = \sum_{S \subseteq e \in E(Y)} \prod_{v \in e \setminus S} X_v.
\]
For every $i \in [K]_0$ let
\[
\mu_i = \max_{S \in \binom{V(Y)}{i}} \E \left[ X_S \right].
\]
Finally, let $\mu = \max_{0 \leq i \leq K} \mu_i$. Here is a consequence of the Main Theorem in \cite{kim2000concentration}.

\begin{thm}\label{thm:kim-vu}
	In the setup above, there exists a constant $D > 0$ such that
	\[
	\Prob \left[ X \geq \mu \log^D (|V(Y)|) \right] = \exp \left( - \Omega \left( \log^2 (|V(Y)|) \right) \right).
	\]
\end{thm}

\begin{proof}[Proof of Claim \ref{clm:G' properties}]
	Recall that we have defined the function
	\[
	 L(\ell,t) = \max \left\{ 1,\frac{(k-1)^\ell (n-2t)}{n^{c+\varepsilon}} \right\} = \max \left\{ 1,\frac{(k-1)^\ell |W_t|}{n^{c+\varepsilon}} \right\}.
	\]
	For brevity, throughout this proof we write
	\[
	L(\ell) \coloneqq L(\ell,t).
	\]
	
	By assumption, for every $v \in W_t$ and every $\ell \leq g-2$, it holds that $P_\ell(G_t,v) \leq L(\ell) \log^C (n)$. Therefore, the number of forbidden edges incident to $v$ does not exceed
	\begin{align*}
	\sum_{\ell=1}^{g-2} P_\ell(G_t,v) & \leq \sum_{\ell=1}^{g-2} L(\ell) \log^C (n) \leq \log^C(n) \sum_{\ell=1}^{g-2} \left( 1 + \frac{(k-1)^\ell |W_t|}{n^{c+\varepsilon}} \right)\\
	& \leq \left( 1 + \frac{|W_t|}{n^\varepsilon} \right) \log^{C+1} (n).
	\end{align*}
	Recalling that $t \leq T_{safe} =  (n - n^\varepsilon)/2$, it follows that $|W_t| \geq n^\varepsilon$. Thus:
	\[
	\sum_{\ell=1}^{g-2} P_\ell(G_t,v) \leq \frac{|W_t|}{n^\varepsilon} \log^{C+2} (n).
	\]
	Therefore:
	\[
	d_{H_t} (v) = |W_t| \pm \frac{|W_t|}{n^\varepsilon} \log^{C+2} (n).
	\]
	By definition, $d_H (v)$ is distributed binomially with parameters $(d_{H_t}(v), p)$. In particular, $\E \left[ d_H(v)  \right] = (1 \pm n^{-\varepsilon} \log^{C+2} (n)) n^\beta$. Applying Chernoff's inequality we obtain:
	\[
	\Prob \left[ \left| d_H(v) - p d_{H_t} (v) \right| > \frac{n^{0.6 \beta}}{2 |W_t|} d_{H_t} (v) \right] \leq \exp \left( - \Omega \left( n^{0.2 \beta} \right) \right).
	\]
	Next apply a union bound to the vertices in $W_t$, and conclude that w.h.p.\ for every $v \in W_t$:
	\[
	d_H(v) = p d_{H_t} (v) \pm \frac{n^{0.6 \beta}}{2 |W_t|} d_{H_t} (v) = \left( 1 \pm n^{-0.4 \beta} \right) n^{\beta},
	\]
	as desired.
	
	We now prove \ref{itm:path count G'}. Suppose that $u,v \in W_t$ are $\ell$-threatened for some $\ell \leq g-2$. Then there is a sequence $u=w_0,\ldots,w_{2m-1}=v$ in $W_t$ witnessing this fact. Now, the number of sequences $w_0,\ldots,w_{2m-1} \in W_t$ such that $\delta_{G_t}(w_0,w_1) + \delta_{G_t}(w_2,w_3) + \ldots + \delta_{G_t}(w_{2m-2},w_{2m-1}) = \ell-m+1$ is bounded from above by
	\[
	\sum_{\ell_1+\ldots+\ell_m = \ell-m+1} \prod_{i=1}^m P_{\ell_i} (G_t).
	\]
	For each such sequence, the probability that all of the edges $w_1w_2,\ldots,w_{2m-3}w_{2m-2}$ are in $E(H)$ is $p^{m-1}$. This yields the following bound:
	\begin{align*}
	\E \left[ T_\ell \right]
	& \leq \sum_{m=1}^{(\ell+1)/2} \sum_{\ell_1+\ldots+\ell_m = \ell-m+1} p^{m-1} \prod_{i=1}^m P_{\ell_i} (G_t)\\
	& \leq \sum_{m=1}^{(\ell+1)/2} \sum_{\ell_1+\ldots+\ell_m = \ell-m+1} p^{m-1} \prod_{i=1}^m \frac{(k-1)^{\ell_i}}{n} |W_t|^2 \log^C(n)\\
	& \leq \frac{(k-1)^\ell}{p} \sum_{m=1}^{(\ell+1)/2} \left(\frac{p |W_t|^2 \log^C(n)}{n}\right)^m \binom{\ell - m}{m-1}\\
	& \leq \frac{(k-1)^\ell}{p} \sum_{m=1}^{(\ell+1)/2} \left(\frac{n^{\beta} |W_t| \ell \log^C(n)}{n}\right)^m
	= O \left( \frac{(k-1)^\ell}{n} |W_t|^2 \log^{C+1}(n) \right).
	\end{align*}
	Hence, by Markov's inequality:
	\[
	\Prob \left[ T_\ell \geq \frac{(k-1)^\ell}{n} |W_t|^2 \log^{C+3} (n) \right] = O \left( \frac{1}{\log^2(n)} \right).
	\]
	Applying a union bound to the $O(\log(n))$ random variables $T_1,\ldots,T_{g-2}$ we conclude that w.h.p., for every $\ell \leq g-2$, it holds that $T_\ell \leq \frac{(k-1)^\ell}{n} |W_t|^2 \log^{C+3} (n)$, as desired.
	
	In order to prove \ref{itm:vertex path count G'}, we must bound $|W_t|(g-2) = O (|W_t| \log (n))$ random variables. For a union bound, it suffices to show that there exists a constant $D = D(C)$ such that for every $v \in W_t$ and $\ell \leq g-2$,
	\[
	\Prob \left[ T_\ell (v) \geq L(\ell) \log^{D} (n) \right] = O \left( \frac{1}{|W_t| \log^2(n)} \right).
	\]
	Let $v \in W_t$, $\ell \leq g-2$, and $m \in \N$. Let $T_\ell^m(v)$ be the number of vertices $u \in W_t$ such that $u,v$ are $\ell$-threatened and there exists a witness with $2m$ vertices. Then $T_\ell(v) \leq \sum_{m=1}^{(\ell+1)/2} T_\ell^m(v)$. We plan to use Theorem \ref{thm:kim-vu} to bound $T_\ell(v)$. However, Theorem \ref{thm:kim-vu} applies only to polynomials with constant degree, whereas $T_\ell(v)$ corresponds to a polynomial with unbounded degree. To get around this, we first show that for large values of $m$, $T_\ell^m(v) = 0$ with sufficiently high probability. Let $M = \lceil (1 - c - \varepsilon - \beta)^{-1} \rceil = O(1)$. We will show that
	\begin{equation}\label{eq:prob bound m>=M edges}
	\Prob \left[ \sum_{m = M}^{(\ell+1)/2} T_\ell^m (v) \geq L(\ell) \log^{C+1} (n) \right] = O \left( \frac{1}{|W_t| \log^2(n)} \right).
	\end{equation}
	We first observe that:
	\begin{equation}\label{eq:prob^M}
	\begin{aligned}
	& \left( \frac{|W_t| n^{\beta} \log^{C+1} (n)}{n} \right)^{M-1} \leq \frac{1}{|W_t|} |W_t|^M n^{-(M-1)(1 - \beta)} \log^{O(1)} (n)\\
	& \leq \frac{1}{|W_t|} n^{M(c+\varepsilon) - (M-1)(1 - \beta)} \log^{O(1)} (n)
	\leq \frac{1}{|W_t|} n^{1 - \beta - M(1-c-\varepsilon - \beta)} \log^{O(1)} (n)\\
	& \leq \frac{n^{-\Omega(1)}}{|W_t|} \leq \frac{1}{|W_t| \log^2(n)}.
	\end{aligned}
	\end{equation}
	Now, by considerations similar to those used to bound $\E \left[T_\ell \right]$:
	\begin{align*}
	\E \left[ \sum_{m = M}^{(\ell+1)/2} T_\ell^m (v) \right] & \leq \sum_{m = M}^{(\ell+1)/2} p^{m-1} \sum_{\ell_1 + \ldots + \ell_m = \ell - m + 1} P_{\ell_1} (G_t, v) \prod_{i=2}^m P_{\ell_i}(G_t)\\
	& \leq \sum_{m=M}^{(\ell+1)/2} L(\ell) \log^{C} (n) \left( \ell p \frac{|W_t|^2 \log^C (n)}{n} \right)^{m-1}\\
	& \leq L(\ell) \log^{C+1} (n) \left( \frac{|W_t| n^{\beta} \log^{C+1} (n)}{n} \right)^{M-1}\\
	& \stackrel{\eqref{eq:prob^M}}{\leq} L(\ell) \log^{C+1} (n) \frac{1}{|W_t| \log^2(n)}.
	\end{align*}
	Inequality \eqref{eq:prob bound m>=M edges} follows from Markov's inequality.
	
	We now use Theorem \ref{thm:kim-vu} to bound $T_\ell^m (v)$, for $m < M$. For every $e \in E(H_t)$, let $X_e$ be the indicator of the event $e \in E(H)$. For $m < M$, let $\mT_\ell^m (v)$ be the collection of \textit{potential} witnesses of length $2m$ to the fact that for some $u \in W_t$, $u,v$ are $\ell$-threatened. In other words, $\mT_\ell^m (v)$ is the collection of sequences $v = w_0,w_1,\ldots,w_{2m-1} \in W_t$ such that:
	\begin{itemize}
		\item For every $1 \leq i \leq m-1$, $w_{2i-1}w_{2i} \in E(H_t)$ and
		
		\item $\delta_{G_t}(w_0,w_1) + \delta_{G_t} (w_2,w_3) + \ldots + \delta_{G_t} (w_{2m-2},w_{2m-1}) = \ell-m+1$.
	\end{itemize}
	
	For conciseness, for $P \in \mT_\ell^m (v)$ and $i \in [m-1]$, we write $e_i(P) = w_{2i-1}w_{2i}$. We also write $\ell_i(P) = \delta_{G_t} (w_{2(i-1)}, w_{2i-1})$. Now consider the polynomial
	\[
	Y_\ell (v) = \sum_{m=1}^{M-1} T_\ell^m (v) = \sum_{m=1}^{M-1} \sum_{P \in \mT_\ell^m (v)} \prod_{i=1}^{m-1} X_{e_i(P)}.
	\]
	We will show that there exists a constant $D > 0$ (independent of $\ell$ and $v$) such that:
	\[
	\Prob \left[ Y_\ell (v) \geq L(\ell) \log^D (n) \right] = \exp \left( - \Omega \left( \log^2 (n) \right) \right).
	\]
	For $S \subseteq E(H_t)$ and $m \in [M-1]$, define the set:
	\[
	\mT_\ell^m(v,S) = \left\{ P \in \mT_\ell^m (v) : S \subseteq \{e_1(P),\ldots,e_{m-1}(P)\} \right\}.
	\]
	Since $\deg1 Y_\ell (v) < M = O(1)$, by Theorem \ref{thm:kim-vu} it suffices to show that there exists a constant $B$ such that for every $S \subseteq E(H_t)$ of cardinality $M-1$ or less, it holds that
	\[
	\E \left[ Y_\ell (v)_S \right] = \sum_{m=1}^{M-1} p^{m-1-|S|} \left| \mT_\ell^m(v,S) \right| = O \left( L(\ell) \log^B(n)\right).
	\]
	For this it is enough to show that for every $S \subseteq E(H_t)$ and every $m < M$:
	\begin{equation}\label{eq:KV goal}
	\left| \mT_\ell^m(v,S) \right| = O \left( p^{|S|-m+1} L(\ell) \log^B (n) \right).
	\end{equation}
	
	Let $S \subseteq E(H_t)$ satisfy $|S| \leq M-1$. We make the following observations: Suppose that for $m < M-1$, $P = w_0,\ldots,w_{2m-1} \in \mT_\ell^m$ satisfies $S \subseteq \{e_1(P),\ldots,e_{m-1}(P)\}$. In particular, $m \geq |S|+1$. Furthermore, there exists an index set $I \in \binom{[m-1]}{|S|}$ such that for every $i \in I$, $e_i(P) \in S$. This implies that for every $i \in I$, $w_{2i}$ is contained in one of the edges in $S$. Therefore,  since $G_t$ is $C$-path-bounded, $w_{2i+1}$ is one of the $O (P_{\ell_i(P)} (G_t, w_{2i})) = O (L(\ell_i(P)) \log^C (n))$ vertices at distance $\ell_i (P)$ from $S$. Using these insights, we may now bound $|\mT_\ell^m(v,S)|$.
	
	Start with the case where $|S| = m-1$. In this case, for every $P = w_0,\ldots,w_{2m-1} \in \mT_\ell^m (v)$, it holds that $S = \{e_1(P),\ldots,e_{m-1}(P) \}$. Therefore, each of $w_1,\ldots,w_{2m-2}$ is contained in an edge in $S$, and the only remaining freedom is in the choice of $w_{2m-1}$. Additionally, $w_{2m-1}$ is at distance at most $\ell$ from $S$. As mentioned, there are $O(L(\ell) \log^C (n))$ such vertices. Therefore, in this case:
	\[
	\left| \mT_\ell^m(v,S) \right| = O \left( L(\ell) \log^C (n) \right),
	\]
	confirming \eqref{eq:KV goal}.
	
	We now assume that $|S| < m-1$. In this case:
	\begin{align*}
	&\left| \mT_\ell^m(v,S) \right|\\
	& \leq \sum_{\ell_1+\ldots+\ell_m = \ell-m+1} P_{\ell_1}(G_t,v) \sum_{I \in \binom{[m-1]}{|S|}} \left(\prod_{i-1 \in I} O \left( L(\ell_i) \log^C (n) \right)\right) \left(\prod_{i-1 \in [m-1] \setminus I} P_{\ell_i}(G_t)\right).
	\end{align*}
	Since $G_t$ is $C$-path-bounded, for every $a \leq g-2$ it holds that
	\[
	P_{a} (G_t) \leq \frac{(k-1)^{a} |W_t|^2 \log^C (n)}{n}.
	\]
	Similarly, $P_a (G_t,v) \leq L(a) \log^C (n) \leq (k-1)^a \log^C (n)$. Therefore, for every ${\ell_1 + \ldots + \ell_m} = \ell-m+1$ and every $I \in \binom{[m-1]}{|S|}$ it holds that:
	\begin{align*}
	P_{\ell_1}(G_t,v) & \left(\prod_{i-1 \in I} O \left( L(\ell_i) \log^C (n) \right)\right) \left(\prod_{i-1 \in [m-1] \setminus I} P_{\ell_i}(G_t)\right)\\
	& = O \left( \log^{Cm} (n) (k-1)^{\ell_1 + \ldots + \ell_m} \left( \frac{|W_t|^2}{n} \right)^{m-1-|S|} \right)\\
	& = O \left( \log^{Cm} (n) (k-1)^\ell \left(\frac{|W_t|^2}{n}\right)^{m-1-|S|} \right).
	\end{align*}
	
	Furthermore, it holds that $\binom{m-1}{|S|} = O(1)$. Thus:
	\begin{align*}
	\left| \mT_\ell^m(v,S) \right| & = O \left( \log^{Cm} (n) \sum_{\ell_1+\ldots+\ell_m = \ell-m+1} (k-1)^{\ell} \left( \frac{|W_t|^2}{n} \right)^{m-1-|S|} \right)\\
	& = O \left( \log^{Cm} (n) \left( \frac{n^\beta |W_t|}{p n} \right)^{m-1-|S|} \ell^m (k-1)^\ell \right)\\
	& = O \left( p^{|S| - m + 1} \log^{(C+1)m} (n) \frac{n^\beta |W_t|}{n} (k-1)^\ell \right)\\
	& = O \left( p^{|S| - m + 1} \frac{(k-1)^\ell |W_t|}{n^{c+\varepsilon}} \right)
	= O \left( p^{|S| - m + 1} L(\ell) \right),
	\end{align*}
	as desired.

	We prove \ref{itm:vertex forbidden count G'} by exposing $E(H)$ in two rounds: Let $v \in W_t$, and let $E(H_t,v)$ denote the set of edges in $E(H_t)$ that are incident to $v$. We first expose $E_1 \coloneqq E(H) \setminus E(H_t,v)$, and then $E_2 \coloneqq E(H) \cap E(H_t,v)$. We note that $E_1$ and $E_2$ are independent. Furthermore, the random variables $Y_1 (v), Y_2(v),\ldots,Y_{g-2} (v)$, as well as the set $W \subseteq W_t$ of vertices $u \in W_t$ such that $u,v$ are $\ell$-threatened are determined by $E_1$. From the proof of \ref{itm:vertex path count G'} it follows that there exists a constant $D = D(C)$ such that $\Prob \left[ |W| \geq \frac{|W_t| \log^{D} (n)}{n^{\varepsilon}} \right] = o \left( |W_t|^{-1} \right)$.
	
	We want to bound $\left| W \cap \Gamma_H (v) \right|$. We observe that given $E_1$ (and hence $W$), $W \cap \Gamma_H(v)$ is a binomial random subset of $W$ with density parameter $p$. Therefore, for any $s \leq \frac{|W_t| \log^{D} (n)}{n^{\varepsilon}}$, conditioned on $|W| = s$, it holds that:
	\begin{align*}
	\Prob \left[ \left| W \cap \Gamma_H(v) \right| \geq \log(n) \right]
	\leq \binom{s}{\log (n)} p^{\log(n)}
	\leq \left( sp \right)^{\log(n)}\\
	\leq \left( \frac{|W_t| \log^D (n)}{n^\varepsilon} \cdot \frac{n^{\beta}}{|W_t|} \right)^{\log (n)} = \exp \left( - \Omega \left( \log^2 (n) \right) \right).
	\end{align*}
	Therefore, using the law of total probability:
	\begin{align*}
	\Prob & \left[ \left| W \cap \Gamma_H(v) \right| \geq \log(n) \right]\\
	 \leq & \Prob \left[ \left| W \cap \Gamma_H(v) \right| \geq \log(n) \Big\vert |W| \leq \frac{|W_t| \log^{D} (n)}{n^\varepsilon} \right] \Prob \left[ |W| \leq \frac{|W_t| \log^{D} (n)}{n^\varepsilon} \right]\\
	& + \Prob \left[ |W| > \frac{|W_t| \log^{D} (n)}{n^\varepsilon} \right] = o \left( \frac{1}{|W_t|} \right).
	\end{align*}
	By applying a union bound to all $|W_t|$ vertices, we conclude that w.h.p., for every $v \in W_t$, $|W \cap \Gamma_H(v) \leq \log (n)$, as desired.
	
	Finally, we prove \ref{itm:G' discrepancy}. For $\emptyset \neq S \subseteq W_t$ such that $|S| \leq |W_t|/n^{\varepsilon/2}$, let $X_S$ be the indicator of the event that $e \left( H[S] \right) \geq |S|n^{0.9\beta}$. Now, $e \left( H[S] \right)$ is distributed binomially with parameters $e \left( H_t [S] \right) \leq |S|^2/2$ and $p$. Hence, by a union bound:
	\begin{align*}
	\E[X_S] = \Prob \left[ e \left( H[S] \right) \geq |S|n^{0.9\beta} \right]
	\leq \binom{|S|^2 / 2}{|S|n^{0.9\beta}} p^{|S|n^{0.9\beta}}.
	\end{align*}
	Applying the inequality $\binom{a}{b} \leq (ea/b)^b$, we obtain:
	\begin{align*}
	\E[X_S] & \leq \left( \frac{e |S|^2 p}{|S| n^{0.9 \beta}} \right)^{|S|n^{0.9\beta}}
	\leq \left( \frac{e|S| n^\beta}{n^{0.9\beta} |W_t|} \right)^{|S|n^{0.9\beta}} \leq \left( \frac{en^{0.1\beta}}{n^{\varepsilon/2}} \right)^{|S| n^{0.9\beta}}.
	\end{align*}
	Applying a union bound over all such sets $S$, we have:
	\begin{align*}
	\E & \left[ \sum_{k=1}^{|W_t|n^{-\varepsilon/2}} \sum_{S \in \binom{W_t}{k}} X_S \right]
	\leq \sum_{k=1}^{|W_t|n^{-\varepsilon/2}} \binom{|W_t|}{k} \left( \frac{en^{0.1\beta}}{n^{\varepsilon/2}} \right)^{k n^{0.9\beta}}\\
	& \leq \sum_{k=1}^{|W_t|n^{-\varepsilon/2}} \left( \frac{e|W_t|}{k} \left( \frac{e}{n^{\varepsilon/2 - 0.1\beta}} \right)^{n^{0.9\beta}} \right)^k
	\leq \sum_{k=1}^{|W_t|n^{-\varepsilon/2}} \left( \left( \frac{2e}{n^{\varepsilon/2 - 0.1\beta}} \right)^{n^{0.9\beta}} \right)^k = \oone.
	\end{align*}
	By Markov's inequality, w.h.p., for every such set $S$, it holds that $X_S = 0$, which is equivalent to \ref{itm:G' discrepancy}.
\end{proof}

\section{Proof of Claim \ref{clm:RGMA in G'}}\label{sec:random matching proof}

We prove Claim \ref{clm:RGMA in G'} by showing that w.h.p.\ various parameters associated with the process $G_t',G_{t+1}',\ldots$ remain close to their expected trajectories. This is motivated by the differential equation method of Wormald \cite{wormald1999differential}, and is similar to the approach taken by Bennett and Bohman \cite{bennett2012natural} in their analysis of random greedy hypergraph matching. As similar results are abundant in the literature, we aim for the simplest exposition and not the sharpest analysis.

We use the following supermartingale inequality of Warnke \cite[Lemma 2.2 and Remark 10]{warnke2016method}. This is a variation on a martingale inequality of Freedman \cite[Theorem 1.6]{freedman1975tail}.

\begin{thm}\label{thm:freedman}
	Let $X_0, X_1, \ldots$ be a supermartingale with respect to a filtration $\mF_0, \mF_1, \ldots$. Suppose that $\left| X_{i+1} - X_i \right| \leq K$ for all $i$, and let $V(j) = \sum_{i=0}^{j-1} \E \left[ \left( X_{i+1} - X_i \right)^2 \given \mF_i \right]$. Then, for any $\lambda,v > 0$,
	\[
	\Prob \left[ X_i > X_0 + \lambda \text{ and } V(i) \leq v \text{ for some } i \right] \leq \exp \left( - \frac{\lambda^2}{2(v+K\lambda/3)} \right).
	\]
\end{thm}

In our application, we find some $v$ such that $V(i) \leq v$ for all $i$. For this $v$, Theorem \ref{thm:freedman} tells us that for every $\lambda > 0$:
\[
\Prob \left[ X_i > X_0 + \lambda \right] \leq \exp \left( - \frac{\lambda^2}{2(v+K\lambda/3)} \right).
\]

We now introduce the random variables we wish to track. Recall that $U_s$ is the set of unsaturated (i.e., degree-$(k-1)$) vertices in $G_{t+s}'$, where $s$ is a non-negative integer. In particular, $U_0 = W_t$ and for every $s \leq T_{freeze}'-t$, it holds that $U_s = W_{t+s}$ and $|U_s| = |W_t| - 2s$. For a vertex $v \in W_t$, let $N(v,s) = \left| \Gamma_H (v) \cap U_s \right|$ be the number of neighbors of $v$ in $H$ that are unsaturated at time $s+t$. We also define the functions
\[
p(s) = 1-\frac{2s}{|W_t|}, \quad n(s) = n^\beta p(s), \quad \varepsilon (s) = \frac{n^{0.6 \beta}}{p(s)^8}.
\]
We observe that for $0 \leq T'_{freeze} - t$ it holds that $p(s) = |U_s|/|W_t|$.

Recall that $t' = (n - |W_t|n^{-\alpha})/2$. We show that w.h.p., for every $v \in W_t$ and every $0 \leq s \leq t'-t$, it holds that
\begin{equation}\label{eq:follows trajectory}
N(v,s) = n(s) \pm \varepsilon (s).
\end{equation}
The guiding intuition is that $\Gamma_H(v) \cap U_s$ behaves like a random subset of $\Gamma_H(v)$ with density $p(s)$. Since by Claim \ref{clm:G' properties} \ref{itm:G' degrees} $\Gamma_H(v) \approx n^\beta$, it follows that $N(v,s)$ is approximated by $n(s)$.

We define the stopping time $\tau$ as the minimum between $t'-t$ and the first time that \eqref{eq:follows trajectory} fails for some $v \in W_t$.

A naive attempt to prove \eqref{eq:follows trajectory} might be to show that $N(v,s) - n(s)$ is a martingale. However, this is not quite true, as the expected one-step change might be non-zero. To remedy this, we consider two \textit{shifted} random variables that are obtained from ${N(v,s)-n(s)}$ and $-(N(v,s) - n(s))$, respectively, by subtracting an error term. These turn out to be supermartingales, enabling us to apply Theorem \ref{thm:freedman}. For every $v \in W_t$, we define:
\begin{align*}
& N^+ (v,s) =
\begin{cases*}
N(v,s) - n(s) - \frac{1}{2} \varepsilon (s) & $s \leq \tau$\\
N^+ (v,s-1) & $s > \tau$
\end{cases*},\\
& N^- (v,s) =
\begin{cases*}
-N(v,s) + n(s) - \frac{1}{2} \varepsilon (s) & $s \leq \tau$\\
N^- (v,s-1) & $s > \tau$
\end{cases*}.
\end{align*}
The fact that these random variables ``freeze'' at time $\tau$ is crucial as it allows us to assume that \eqref{eq:follows trajectory} holds when calculating the maximal and expected one-step changes.

We say that $uv \in E(H)$ is \termdefine{available at time $s$} if $u,v \in U_s$ and $\delta_{G_{t+s}'} (u,v) \geq g-1$. For $v \in U_s$, let $A(v,s)$ be the set of available edges at time $s$ that are incident to $v$. We note that $A(v,s) \subseteq \{ vu : u \in \Gamma_H(v) \cap U_s \}$ (the cardinality of this last set is $N(v,s)$). In general, this inclusion might be strict because there may be vertices $u\in \Gamma_H(v) \cap U_s$ with $\delta_{G_{t+s}'} (u,v) < g-1$ (however, it is true that $A(v,0) = \{ vu : u \in \Gamma_H(v) \cap U_0 \}$). Nevertheless, the difference between the sets is small.

\begin{clm}\label{clm: most neighbors are available}
	For every $0 \leq s \leq t'-t$ and every $v \in U_s$, it holds that $|A(v,s)| \geq N(v,s) - \log (n)$. Consequently, if $\tau > s$, then $|A(v,s)| = n(s) \pm 1.1 \varepsilon(s)$.
\end{clm}

\begin{proof}
	By assumption, $G'$ satisfies the conclusions of Claim \ref{clm:G' properties}. In particular, \ref{itm:vertex forbidden count G'} implies that there are at most $\log (n)$ vertices $u \in U_{s} \subseteq W_t$ such that $u \in \Gamma_H(v)$ and such that $u,v$ are $\ell$-threatened for some $\ell < g-1$. Since $G_{t+s}'$ is a subgraph of $G'$, this holds for $G_{t+s}'$ as well, and the claim follows.
	
	We now observe that for every $s \leq t'-t$, $\varepsilon(s) \geq \varepsilon(0) = n^{0.6\beta} > 10 \log (n)$. Therefore, if $\tau > s$:
	\[
	|A(v,s)| = N(v,s) \pm \log(n) = n(s) \pm \left( \varepsilon(s) + \log (n) \right) = n(s) \pm 1.1 \varepsilon(s).
	\]
\end{proof}

\begin{clm}\label{clm:supermartingale}
	For every $v \in W_t$, the sequences $\{N^+(v,s)\}_{s=0}^\infty$ and $\{N^-(v,s)\}_{s=0}^\infty$ are supermartingales with respect to the filtration induced by $\{ G_{t+s}' \}_{s=0}^\infty$. Furthermore, for every $0 \leq s \leq t'-t$, it holds that
	\[
	\E \left[ \left| N^-(v,s+1) - N^-(v,s) \right| \right], \E \left[ \left| N^+(v,s+1) - N^+(v,s) \right| \right] \leq \frac{5n^\beta}{|W_t|}.
	\]
\end{clm}

\begin{proof}
	We show that $\{N^+(v,s)\}_{s = 0}^\infty$ is a supermartingale for every $v \in W_t$. The proof for $\{N^-(v,s)\}_{s = 0}^\infty$ is similar. We need to show that for every $s \geq 0$, it holds that
	\begin{equation}\label{eq:supermartingale}
	\E \left[ N^+(v,s+1) - N^+(v,s) \given G_t',G_{t+1}',\ldots, G_{t+s}' \right] \leq 0.
	\end{equation}
	We apply the law of total probability. We first observe that if $\tau \leq s$, then, by definition, $N^+(v,s+1) = N^+(v,s)$, so \eqref{eq:supermartingale} holds. It thus suffices to show that
	\[
	\E \left[ N^+(v,s+1) - N^+(v,s) \given G_{t+s}' \land \tau \geq s+1  \right] \leq 0.
	\]
	We therefore assume that $\tau \geq s+1$. By Claim \ref{clm: most neighbors are available} this implies that for every $u \in U_s$, it holds that $|A(u,s)| = n(s) \pm 1.1 \varepsilon (s)$. In particular, this holds for every vertex in $\Gamma_H(v) \cap U_s$. Finally, the number of available edges is equal to $|U_s| \left( n(s) \pm 1.1 \varepsilon(s) \right)/2$. Therefore:
	\begin{align*}
	\E & \left[ N(v,s+1) - N(v,s) \given G_{t+s}' \land \tau \geq s+1 \right] = - \frac{2}{|U_s|\left( n(s) \pm 1.1 \varepsilon(s) \right)} \sum_{u \in \Gamma_H(v) \cap U_s} |A(u,s)|\\
	& = - \frac{2 N(v,s) (n(s) \pm 1.1 \varepsilon (s))}{|U_s| (n(s) \pm 1.1 \varepsilon (s)}
	= - \frac{2 (n(s) \pm 1.1 \varepsilon (s))^2}{|U_s| (n(s) \pm 1.1 \varepsilon (s))}
	= - \frac{2 n(s)}{|U_s|} \left( 1 \pm \frac{3.5 \varepsilon (s)}{n(s)} \right)\\
	& = - \frac{2n^\beta}{|W_t|} \pm \frac{7 \varepsilon (s)}{|U_s|}.
	\end{align*}
	Next, we observe that:
	\begin{align*}
	n(s+1) - n(s) = - \frac{2n^\beta}{|W_t|}.
	\end{align*}
	Finally, we note that:
	\begin{align*}
	\varepsilon(s+1) - \varepsilon(s) & = n^{0.6 \beta} \left( \frac{1}{p(s+1)^8} - \frac{1}{p(s)^8} \right) = \frac{n^{0.6\beta}}{p(s)^8} \left( \frac{p(s)^8}{p(s+1)^8} - 1 \right)\\
	& = \varepsilon (s) \left( \left(\frac{|W_t|-2s}{|W_t|-2s-2}\right)^8 - 1 \right)\\
	& = \varepsilon(s) \left( \left( 1 - \frac{2}{|U_s|} \right)^{-8} - 1 \right).
	\end{align*}
	Hence, by Taylor's Theorem (recalling that $|U_s| \geq n^{-\alpha} |W_t| = \omegaone$):
	\[
	\varepsilon(s+1) - \varepsilon(s) = \varepsilon(s) \left( \frac{16}{|U_s|} + O \left( \frac{1}{|U_s|^2} \right) \right) \in \left[ \frac{16 \varepsilon(s)}{|U_s|}, \frac{18 \varepsilon (s)}{|U_s|} \right].
	\]
	Therefore:
	\begin{align*}
	& \E \left[ N^+(v,s+1) - N^+(v,s) \given G_{t+s}' \land \tau \geq s+1 \right]\\
	& = \E \left[ N(v,s+1) - N(v,s) \given G_{t+s}' \land \tau \geq s+1 \right] - \left( n(s+1) - n(s) \right) - \frac{1}{2} \left( \varepsilon(s+1) - \varepsilon(s) \right)\\
	& \leq \left(- \frac{2n^\beta}{|W_t|} \pm \frac{7 \varepsilon (s)}{|U_s|}\right) + \frac{2n^\beta}{|W_t|} - \frac{8\varepsilon(s)}{|U_s|} \leq 0,
	\end{align*}
	as desired.
	
	We also observe that the estimates above imply:
	\begin{align*}
	\E & \left[ \left| N^+(v,s+1) - N^+(v,s) \right| \right]\\
	& \leq \E \left[ N(v,s) - N(v,s+1) \right] + n(s) - n(s+1) + \frac{1}{2} \left(\varepsilon(s+1) - \varepsilon(s)\right)\\
	& \leq \frac{2n^\beta}{|W_t|} + \frac{7 \varepsilon (s)}{|U_s|} + \frac{2n^\beta}{|W_t|} + \frac{9 \varepsilon(s)}{|U_s|} \leq \frac{4 n^\beta}{|W_t|} + \frac{16 \varepsilon (t'-t)}{|U_{t'-t}|} = \frac{4 n^\beta}{|W_t|} + \frac{16 n^{0.6\beta}}{n^{-8\alpha} |W_t|}.
	\end{align*}
	Recalling that $\alpha = \alphadef$, we obtain:
	\[
	\E \left[ \left| N^+(v,s+1) - N^+(v,s) \right| \right] \leq \frac{5 n^\beta}{|W_t|},
	\]
	as claimed.
\end{proof}

In order to apply Theorem \ref{thm:freedman}, we first note that the maximal one-step change in $N(v,s)$ is $2$. Furthermore, for every $s \leq t'-t$, $|n(s+1)-n(s)|, \varepsilon(s+1) - \varepsilon(s) = \oone$. Therefore, the maximal one-step change in $N^+(v,s)$ and $N^-(v,s)$ is bounded from above by $3$. Hence, for every $1 \leq s \leq t'-t$, it holds that:
\begin{align*}
V(s) & \leq V(t'-t) \leq 3 \sum_{i=0}^{t'-t-1} \E \left[ \left| N^+(v,i+1) - N^+(v,i) \right| \given G_{t+i}' \right]\\
& \leq 3 (t'-t) \frac{5n^\beta}{|W_t|} = O \left( n^\beta \right).
\end{align*}
 By applying Theorem \ref{thm:freedman} with $K=3$, $\lambda = \varepsilon(s) / 2$ and $v = n^\beta \log(n)$, we conclude that for every $w \in W_t$ and every $0 \leq s \leq t'-t$:
\[
\Prob \left[ N^+(w,s) \geq \varepsilon(s)/2 \right] \leq \exp \left( - \frac{\varepsilon(s)^2/4}{2 (n^\beta \log (n) + 3\varepsilon(s)/2)} \right) = \exp \left( - \Omega \left( n^{\beta/100} \right) \right).
\]
Applying a union bound to the $O \left( |W_t|^2 \right)$ choices for $v$ and $s$, we conclude that w.h.p., for every $v \in W_t$ and $0 \leq s \leq t'-t$, it holds that $N^+(v,s) \leq \frac{1}{2}\varepsilon(s)$. A similar calculation implies the analogous result for $N^-(v,s)$. This implies that $\tau \geq t'-t$. therefore, w.h.p., for every $v \in W_t$ and $0 \leq s \leq t'-t$, it holds that
\[
N(v,s) = n(s) \pm \varepsilon(s).
\]
In particular, this implies that $T_{freeze}' \geq t'$. Therefore, $G_{t'} = G_{t'}'$. Thus $U_{t'-t} = W_{t'}$.

In order to show that $G_{t'}$ is path-bounded we estimate the probability that a given set of vertices is in $U_{t'-t}$.

\begin{clm}\label{clm:vertex cover probability}
	Let $A \subseteq W_t$ satisfy $|A| \leq |W_t| / n^{\varepsilon/2}$. Then:
	\[
	\Prob \left[ A \subseteq U_{t'-t} \right] = \oneoone n^{- \alpha |A|}.
	\]
\end{clm}

\begin{proof}
	Similar to the proof of Claim \ref{clm:subset containment bound with times}, we denote by $B_s$ the event that $A \subseteq U_s$, and observe that:
	\[
	\Prob \left[ A \subseteq U_{t'-t} \right] = \Prob \left[ B_1 \right] \times \Prob \left[ B_2 \given B_1 \right] \times \ldots \times \Prob \left[ B_{t'-t} \given B_{t'-t-1} \right].
	\]
	Using the law of total probability, for every $s \leq t'-t$ it holds that
	\[
	\Prob \left[ B_s \given B_{s-1} \right] = \Prob \left[ B_s \given B_{s-1} \land \tau \geq s \right] \Prob \left[ \tau \geq s \right] + \Prob \left[ B_s \given B_{s-1} \land \tau < s \right] \Prob \left[ \tau < s \right].
	\]
	Now, if $\tau \geq s$, then every vertex in $U_{s-1}$ is incident to $\left( 1 \pm n^{-0.3\beta} \right) n(s)$ available edges, and there are $\left( 1 \pm n^{-0.3\beta} \right) |U_{s}|n(s)/2$ available edges in total. Furthermore, by assumption, $H$ satisfies Claim \ref{clm:G' properties} \ref{itm:G' discrepancy}. Therefore $e \left( H[A] \right) \leq |A| n^{0.9\beta}$. Hence $A$ is incident to $\left( 1 \pm n^{-0.3\beta} \right)|A|n(s) \pm |A| n^{0.9\beta} = \left( 1 \pm n^{-0.05\beta} \right)|A|n(s)$ available edges. Thus:
	\[
	\Prob \left[ B_s \given B_{s-1} \land \tau \geq s \right] = \left( 1 - \frac{\left( 1 \pm n^{-0.05\beta} \right)|A|n(s)}{\left( 1 \pm n^{-0.3\beta} \right) |U_s|n(s)/2} \right) = \left( 1 - \left( 1 \pm n^{-0.04\beta} \right) \frac{2|A|}{|U_s|} \right).
	\]
	Since $\Prob \left[ \tau \leq t' \right] = \exp \left( -\Omega (n^{\beta/100}) \right)$, we conclude that:
	\[
	\Prob \left[ B_s \given B_{s-1} \right] = \left( 1 - \left( 1 \pm n^{-0.04\beta} \right) \frac{2|A|}{|U_s|} \right) \left( 1 \pm \exp \left( -\Omega \left( n^{\beta/100} \right) \right) \right).
	\]
	Thus:
	\[
	\Prob \left[ A \subseteq U_{t'-t} \right] = \prod_{s=1}^{t'-t} \Prob \left[ B_s \given B_{s-1} \right] = \oneoone n^{-\alpha |A|}.
	\]
\end{proof}

We can now use Markov's inequality to show that $G_{t'}$ satisfies Definition \ref{def:path bounded} \ref{itm:global path bound}. Recall that for $\ell \leq g-2$, $T_\ell$ is the number of $\ell$-threatened pairs in $G'$. By Claim \ref{clm:G' properties} \ref{itm:path count G'}, $T_\ell \leq \frac{(k-1)^\ell}{n}|W_t|^2 \log^Q (n)$. By Claim \ref{clm:vertex cover probability}, the expected number of these pairs that are also contained in $U_{t'-t}$ is at most $\oneoone n^{-2\alpha} T_\ell \leq \frac{(k-1)^\ell}{n} |W_t|^2 n^{-2\alpha} \log^Q (n) = \frac{(k-1)^\ell}{n} |W_{t'}|^2 \log^Q (n)$. Applying Markov's inequality and a union bound, we conclude that w.h.p., for every $\ell \leq g-2$, $P_\ell (G_{t'}) \leq \frac{(k-1)^\ell}{n} |W_{t'}|^2 \log^{Q+2} (n)$.

Finally, we show that $G_{t'}$ satisfies Definition \ref{def:path bounded} \ref{itm:local path count} . Let $v \in W_t$ and let $\ell \leq g-2$. Let $A = A_\ell (v)$ be the set of vertices $u \in W_t$ such that $u,v$ is $\ell$-threatened in $G'$. Then $|A| = T_\ell(v)$, and by assumption $T_\ell(v) \leq L(\ell,t) \log^Q (n)$. We will bound the probability that $|A \cap W_{t'}| \geq L(\ell,t') \log^{Q+1} (n)$. By Claim \ref{clm:vertex cover probability}, for every $B \in \binom{A}{L(\ell,t') \log^{Q+1} (n)}$, it holds that:
\[
\Prob \left[ B \subseteq W_{t'} \right] = \oneoone n^{-\alpha L(\ell,t') \log^{Q+1} (n)}.
\]
Therefore, by a union bound:
\[
\Prob \left[ |A \cap W_{t'} | \geq L(\ell,t') \log^{Q+1} (n) \right] \leq \binom{|A|}{L(\ell,t') \log^{Q+1} (n)} \oneoone n^{-\alpha L(\ell,t') \log^{Q+1} (n)}.
\]
Applying the inequality $\binom{a}{b} \leq (ea/b)^b$, it follows that:
\[
\Prob \left[ |A \cap W_{t'} | \geq L(\ell,t') \log^{Q+1} (n) \right] \leq \oneoone \left( \frac{e |A|}{n^\alpha L(\ell,t') \log^{Q+1} (n)} \right)^{L(\ell,t') \log^{Q+1} (n)}.
\]
Observing that $L(\ell,t) \leq n^\alpha L(\ell,t')$, and that $|A| \leq L(\ell,t)\log^Q (n)$, we have:
\begin{align*}
\Prob \left[ |A \cap W_{t'} | \geq L(\ell,t') \log^{Q+1} (n) \right] & \leq \oneoone \left( \frac{e L(\ell,t) \log^Q (n)}{n^\alpha L(\ell,t') \log^{Q+1} (n)} \right)^{L(\ell,t') \log^{Q+1} (n)}\\
& \leq \oneoone \left( \frac{e}{\log (n)} \right)^{L(\ell,t') \log^{Q+1} (n)} = n^{-\omegaone}.
\end{align*}
We apply a union bound over the $O(|W_t| \log (n))$ choices of $v$ and $\ell$ to conclude that w.h.p., for every $v \in W_{t'}$ and every $\ell \leq g-2$, it holds that $|A_\ell (v) \cap W_{t'}| \leq L(\ell,t') \log^{Q+1} (n)$. Therefore, for $D = Q+2$, w.h.p.\ $G_{t'}$ is $D$-path-bounded.

\section{Counting high-girth graphs: proof of Theorem \ref{thm:enumeration}}\label{sec:enumeration}

Let $k,c,n$ be as in the statement of Theorem \ref{thm:enumeration}. Let $g = c \log_{k-1}(n)$. We prove the theorem by considering the number of (labeled) graphs that can be produced by the $(G,g,k)$-high-girth process, with $G$ a random Hamilton cycle on $n$ vertices. First, there are $n!/(2n)$ choices for the Hamilton cycle $G$. In Lemma \ref{lem:initial phase} we showed that if, for $d \geq 3$, $G'$ is a $(d-1)$-regular graph on $n$ vertices, then for every $0 \leq t \leq T \coloneqq (n-n^{c+\varepsilon})/2$, the $(G',g,d)$-high-girth process has $(1-\oone)(n-2t)^2/2$ available edges. Thus, the total number of choices in this phase is equal to
\begin{align*}
N(d) \coloneqq \prod_{t=0}^T \left( (1-\oone) \frac{(n-2t)^2}{2} \right) & = \left( (1-\oone) \frac{n^2}{2} \right)^{T+1} \prod_{t=0}^T \left( 1 - \frac{2t}{n} \right)^2\\
& = \left( (1 - \oone) \frac{n^2}{2e^2} \right)^{n/2}
\end{align*}
By Proposition \ref{prop: induction step}, w.h.p.\ the $(G',g,d)$-high-girth-process succeeds in constructing a $d$-regular graph. Therefore the number of successful runs of the algorithm is at least $(1-\oone) N(d)$. Returning to the $(G,g,k)$-high-girth-process, we conclude that the number of successful runs for this algorithm is at least
\[
(1 - \oone) \frac{n!}{2n} N(3) \times N(4) \times \ldots \times N(k) = \left( \oneoone \frac{n}{e} \right)^n \left( (1-\oone) \frac{n^2}{2e^2} \right)^{(k-2)n/2}.
\]

Let $H$ be one of the $k$-regular graphs that the $(G,g,k)$-high-girth-process can produce. Then $H$ is the disjoint union of the Hamilton cycle $G$ and the $(k-2)$-regular graph $H'$ of the chords chosen by the process. There are fewer than $e(H')!$ ways in which the process can construct $H'$, according to the order in which the edges of $H'$ are added. Furthermore, since $H$ is $k$-regular, it contains fewer than $k^n$ Hamilton cycles. This serves as an upper bound on the number of possible choices for $G$. Since $e(H') = (k-2)n/2$, the algorithm can produce at least
\[
\frac{\left( \oneoone n/e \right)^n \left( (1-\oone) n^2/ \left( 2e^2 \right) \right)^{(k-2)n/2}}{k^n \left( (k-2)n/2 \right)!} = \left( \Omega (n) \right)^{kn/2}
\]
different graphs, as claimed.

We remark that the enumeration can be improved by employing better estimates of the number of Hamiltonian cycles and one-factorizations in regular graphs. However, even for the case of cubic graphs, it seems unlikely that the best known bounds \cite{gebauer2011enumerating} yield a tight result.

\section{Concluding remarks and open problems}\label{sec:closing}

\begin{itemize}
	\item A natural and interesting variation of our algorithm starts with $n$ isolated vertices rather than a Hamilton cycle. At each step we add a uniformly chosen edge subject to the constraints that all vertex degrees remain $\le k$ and the girth remains $\ge g$ (this is in contrast to our own process where we only connect vertices of \textit{minimum} degree). Ruci\'nski and Wormald \cite{rucinski1992random} studied this process without the girth constraint, and showed that w.h.p.\ the process yields a regular graph. We believe that ideas from the present work can be modified to show that even for $g = c \log_{k-1} (n)$ (with $c<1$) the process is likely to produce a $k$-regular graph. However, new complications arise, which presumably require Wormald's differential equation method. We leave this to future work.
	
	\item To what extent do our graphs resemble random regular graphs? Numerical experiments that we have conducted suggest that they are Ramanujan, or at least nearly Ramanujan. For reference recall Friedman's famous result \cite{friedman2008proof} that almost all regular graphs are nearly Ramanujan.
	
	\item The large-scale geometry of graphs holds many open questions. Thus, it is not hard to show that every $n$-vertex $k$-regular graph of girth $g$ has at most ${\frac{n k}{g} (k-1)^{g/2}}$ cycles of length $g$. On the other hand in LPS graphs the number is $\tilde{\Omega}(n^{4/3})$ \cite{bourque2019kissing}, and no graphs are known for which this number is larger. Our numerical calculations suggest that in our graphs this number is in fact $\Theta_k \left( (k-1)^g / g \right)$. It would also be interesting to determine the smallest $\gamma = \gamma(n,k,g)$ such that every girth-$g$ $k$-regular graph on $n$ vertices has a set of $\gamma$ edges that intersects every $g$-cycle.\\ The possible relation between a graph's girth and its diameter is particularly intriguing. It follows from \cite{erdos1963regulare} and Moore's bound that
	\[
	2 \ge \limsup \frac{\girth(G)}{\diam(G)}\ge 1,
	\]
	where the $\limsup$ ranges over all graphs where all vertex degrees are $\ge 3$. Nothing better seems to be known at the moment.\\ Even more remarkably, we do not know whether
	\[
	\sup (\girth(G) -  \diam(G))
	\]
	is finite or not. The $\sup$ is over all $G$ in which all vertex degrees are $\ge 3$.
\end{itemize}

\bibliography{ham_cycle}

\providecommand{\bysame}{\leavevmode\hbox to3em{\hrulefill}\thinspace}
\providecommand{\MR}{\relax\ifhmode\unskip\space\fi MR }
\providecommand{\MRhref}[2]{%
  \href{http://www.ams.org/mathscinet-getitem?mr=#1}{#2}
}
\providecommand{\href}[2]{#2}
\begin{thebibliography}{10}

\bibitem{bayati2009generating}
Mohsen Bayati, Andrea Montanari, and Amin Saberi, \emph{Generating random
  graphs with large girth}, Proceedings of the twentieth annual ACM-SIAM
  symposium on Discrete algorithms, Society for Industrial and Applied
  Mathematics, 2009, pp.~566--575.

\bibitem{bennett2012natural}
Patrick Bennett and Tom Bohman, \emph{A natural barrier in random greedy
  hypergraph matching}, Combinatorics, Probability and Computing (2012), 1--10.

\bibitem{biggs1983sextet}
NL~Biggs and MJ~Hoare, \emph{The sextet construction for cubic graphs},
  Combinatorica \textbf{3} (1983), no.~2, 153--165.

\bibitem{biggs1998constructions}
Norman Biggs, \emph{Constructions for cubic graphs with large girth}, the
  Electronic Journal of Combinatorics \textbf{5} (1998), no.~1, 1.

\bibitem{bohman2009triangle}
Tom Bohman, \emph{The triangle-free process}, Advances in Mathematics
  \textbf{221} (2009), no.~5, 1653--1677.

\bibitem{bohman2018large}
Tom Bohman and Lutz Warnke, \emph{Large girth approximate {Steiner} triple
  systems}, Journal of the London Mathematical Society (2018).

\bibitem{bollobas2000constrained}
B{\'e}la Bollob{\'a}s and Oliver Riordan, \emph{Constrained graph processes},
  The Electronic Journal of Combinatorics \textbf{7} (2000), no.~1, R18.

\bibitem{bourque2019kissing}
Maxime~Fortier Bourque and Bram Petri, \emph{Kissing numbers of regular
  graphs}, arXiv preprint arXiv:1909.12817 (2019).

\bibitem{chandran2003high}
L~Sunil Chandran, \emph{A high girth graph construction}, SIAM journal on
  Discrete Mathematics \textbf{16} (2003), no.~3, 366--370.

\bibitem{chiu1992cubic}
Patrick Chiu, \emph{Cubic {Ramanujan} graphs}, Combinatorica \textbf{12}
  (1992), no.~3, 275--285.

\bibitem{dahan2014regular}
Xavier Dahan, \emph{Regular graphs of large girth and arbitrary degree},
  Combinatorica \textbf{34} (2014), no.~4, 407--426.

\bibitem{erdos1963regulare}
Paul Erd{\H{o}}s and Horst Sachs, \emph{Regul{\"a}re graphen gegebener
  taillenweite mit minimaler knotenzahl}, Wiss. Z. Martin-Luther-Univ.
  Halle-Wittenberg Math.-Natur. Reihe \textbf{12} (1963), no.~251-257, 22.

\bibitem{erdos1995size}
Paul Erd{\"o}s, Stephen Suen, and Peter Winkler, \emph{On the size of a random
  maximal graph}, Random Structures \& Algorithms \textbf{6} (1995), no.~2-3,
  309--318.

\bibitem{freedman1975tail}
David~A Freedman, \emph{On tail probabilities for martingales}, the Annals of
  Probability \textbf{3} (1975), no.~1, 100--118.

\bibitem{friedman2008proof}
Joel Friedman, \emph{A proof of {Alon's} second eigenvalue conjecture and
  related problems}, American Mathematical Soc., 2008.

\bibitem{gamburd2009girth}
Alexander Gamburd, Shlomo Hoory, Mehrdad Shahshahani, Aner Shalev, and Balint
  Vir{\'a}g, \emph{On the girth of random {Cayley} graphs}, Random Structures
  \& Algorithms \textbf{35} (2009), no.~1, 100--117.

\bibitem{gebauer2011enumerating}
Heidi Gebauer, \emph{Enumerating all {Hamilton} cycles and bounding the number
  of {Hamilton} cycles in 3-regular graphs}, The Electronic Journal of
  Combinatorics \textbf{18} (2011), no.~1, P132.

\bibitem{glock2016existence}
Stefan Glock, Daniela K{\"u}hn, Allan Lo, and Deryk Osthus, \emph{The existence
  of designs via iterative absorption}, arXiv preprint arXiv:1611.06827 (2016).

\bibitem{glock2018conjecture}
\bysame, \emph{On a conjecture of {Erd{\H{o}}s} on locally sparse {Steiner}
  triple systems}, arXiv preprint arXiv:1802.04227 (2018).

\bibitem{grable1997random}
David~A Grable, \emph{On random greedy triangle packing}, the Electronic
  Journal of Combinatorics \textbf{4} (1997), no.~1, P11.

\bibitem{hoory2002graphs}
Shlomo Hoory, \emph{On graphs of high girth}, Ph.D. thesis, Hebrew University
  of Jerusalem Israel, 2002.

\bibitem{keevash2014existence}
Peter Keevash, \emph{The existence of designs}, arXiv preprint arXiv:1401.3665
  (2014).

\bibitem{kim2000concentration}
Jeong~Han Kim and Van~H Vu, \emph{Concentration of multivariate polynomials and
  its applications}, Combinatorica \textbf{20} (2000), no.~3, 417--434.

\bibitem{krivelevich2018random}
Michael Krivelevich, Matthew Kwan, Po-Shen Loh, and Benny Sudakov, \emph{The
  random k-matching-free process}, Random Structures \& Algorithms \textbf{53}
  (2018), no.~4, 692--716.

\bibitem{kwan2016almost}
Matthew Kwan, \emph{Almost all {Steiner} triple systems have perfect
  matchings}, arXiv preprint arXiv:1611.02246 (2016).

\bibitem{lazebnik1995new}
Felix Lazebnik, Vasiliy~A Ustimenko, and Andrew~J Woldar, \emph{A new series of
  dense graphs of high girth}, Bulletin of the American mathematical society
  \textbf{32} (1995), no.~1, 73--79.

\bibitem{lubotzky1988ramanujan}
Alexander Lubotzky, Ralph Phillips, and Peter Sarnak, \emph{Ramanujan graphs},
  Combinatorica \textbf{8} (1988), no.~3, 261--277.

\bibitem{mckay2004short}
Brendan~D McKay, Nicholas~C Wormald, and Beata Wysocka, \emph{Short cycles in
  random regular graphs}, The Electronic Journal of Combinatorics \textbf{11}
  (2004), no.~1, 66.

\bibitem{morgenstern1994existence}
Moshe Morgenstern, \emph{Existence and explicit constructions of $q+1$ regular
  {Ramanujan} graphs for every prime power q}, Journal of Combinatorial Theory,
  Series B \textbf{62} (1994), no.~1, 44--62.

\bibitem{osthus2001random}
Deryk Osthus and Anusch Taraz, \emph{Random maximal {H}-free graphs}, Random
  Structures \& Algorithms \textbf{18} (2001), no.~1, 61--82.

\bibitem{picollelli2011final}
Michael~E Picollelli, \emph{The final size of the {$C_4$}-free process},
  Combinatorics, Probability and Computing \textbf{20} (2011), no.~6, 939--955.

\bibitem{picollelli2014final}
\bysame, \emph{The final size of the {$C_\ell$}-free process}, SIAM Journal on
  Discrete Mathematics \textbf{28} (2014), no.~3, 1276--1305.

\bibitem{rucinski1992random}
Andrzej Ruci{\'n}ski and Nicholas~C Wormald, \emph{Random graph processes with
  degree restrictions}, Combinatorics, Probability and Computing \textbf{1}
  (1992), no.~2, 169--180.

\bibitem{warnke2014c_ell}
Lutz Warnke, \emph{The {$C_\ell$}-free process}, Random Structures \&
  Algorithms \textbf{44} (2014), no.~4, 490--526.

\bibitem{warnke2016method}
\bysame, \emph{On the method of typical bounded differences}, Combinatorics,
  Probability and Computing \textbf{25} (2016), no.~2, 269--299.

\bibitem{weiss1984girths}
Alfred Weiss, \emph{Girths of bipartite sextet graphs}, Combinatorica
  \textbf{4} (1984), no.~2-3, 241--245.

\bibitem{wormald1999differential}
Nicholas~C Wormald, \emph{The differential equation method for random graph
  processes and greedy algorithms}, Lectures on approximation and randomized
  algorithms \textbf{73} (1999), 73--155.

\end{thebibliography}
\bibliographystyle{amsplain}
	
\end{document}